\documentclass[12pt, reqno]{amsart}

\usepackage[hmargin=0.8in,height=8.6in]{geometry}
\usepackage{amssymb,amsthm, times}
\usepackage{delarray,verbatim}
\usepackage{ifpdf}
\ifpdf
\usepackage[pdftex]{graphicx}
 \else
\usepackage[dvips]{graphicx}
 \fi

\usepackage{bm}

\usepackage{ifpdf}
\usepackage{color}
\definecolor{webgreen}{rgb}{0,.5,0}
\definecolor{webbrown}{rgb}{.8,0,0}
\definecolor{emphcolor}{rgb}{0.95,0.95,0.95}

\usepackage[pagebackref]{hyperref}
\usepackage{hyperref}
\hypersetup{%
          colorlinks=true,
          linkcolor=webbrown,
          filecolor=webbrown,
          citecolor=webgreen,
          breaklinks=true}
\ifpdf \hypersetup{pdftex,
            pdfstartview=FitH, 
            bookmarksopen=true,
            bookmarksnumbered=true
} \else \hypersetup{dvips} \fi

\newcommand {\costb}{\gamma_S}
\newcommand {\costs}{\gamma_I}

\renewcommand{\S}{\mathcal{S}}

\newcommand{\lapinv}{\Phi(q)}

\newcommand {\pcheck}{p}
\newcommand {\acheck}{1}
\newcommand {\lap}{\zeta}

\numberwithin{equation}{section}

\newtheorem{remark}{Remark}[section]
\newtheorem{lemma}{Lemma}[section]
\newtheorem{example}{Example}[section]
\newtheorem{assump}{Assumption}[section]
\newtheorem{definition}{Definition}[section]
\numberwithin{remark}{section} \numberwithin{proposition}{section}
\numberwithin{corollary}{section}
\newcommand {\R}{\mathbb{R}}

\newcommand {\F}{\mathcal{F}}

\newcommand {\p}{\mathbb{P}}

\newcommand {\E}{\mathbb{E}}

\newcommand{\diff}{{\rm d}}

\newcommand{\lev}{L\'{e}vy }

\title[Optimality of two-parameter strategies in stochastic control]{Optimality of two-parameter strategies in stochastic control}
\thanks{This version: \today. }
\thanks{$*$\, Department of Mathematics,
Faculty of Engineering Science, Kansai University, Suita-shi, Osaka 564-8680, Japan. Email: \mbox{{\em
kyamazak@kansai-u.ac.jp.}} Phone: +81-6-6368-1527.  }
\author[K. Yamazaki]{Kazutoshi Yamazaki$^*$}
\date{}
\begin{document}

\begin{abstract}
In this note, we study a class of stochastic control problems where
the optimal strategies are described by two parameters.  These include
a subset of singular control, impulse control, and two-player
stochastic games.  The parameters are first chosen by the two
continuous/smooth fit conditions, and then the optimality of the
corresponding strategy is shown by verification arguments. Under the
setting driven by a spectrally one-sided \lev process, these
procedures can be efficiently done thanks to the recent developments of
scale functions. In this note, we illustrate these techniques using
several examples where the optimal strategy as well as the value
function can be concisely expressed via scale functions.
\\
\noindent \small{\noindent  AMS 2010 Subject Classifications: 60G51, 93E20, 49J40 
}\\
\end{abstract}

\maketitle

\section{Introduction}

In stochastic control, one wants to optimally control a stochastic process so as to minimize or maximize the expected value of a given payoff that is determined by the paths of the control and/or controlled processes. In other words, one wants to find an \emph{optimal strategy} that attains the minimal or maximal expected value, called the \emph{(optimal) value function}.  Essentially all real-life phenomena contain uncertainty, and consequently the problem of stochastic control arises everywhere.  It is well-studied in, among others, finance (e.g.\ portfolio optimization, asset pricing, risk management), economics (search, real options, games), insurance, inventory management, and queues. 

Because it has a wide range of applications and is studied in a variety of fields, there are many different approaches for modeling.  A model can be categorized by (i) discrete/continuous time, (ii) discrete/continuous state,  and (iii) finite/infinite horizon. Except for very special cases, the only case one can expect an analytical solution is the continuous-time, continuous-state model with the infinite horizon. For other cases, one typically needs to rely on numerical approaches, such as value/policy iterations, backward inductions, and finite difference methods. See, e.g., Puterman \cite{puterman2014markov}.

In this note, we focus on a relatively simple class of stochastic control where analytical solutions can be obtained.  We assume the continuous-time, infinite-horizon case with the state space given by $\R$ or its subset.  In addition, randomness is assumed to be modeled by a one-dimensional \emph{spectrally one-sided \lev process}, or a \lev process with only one-sided jumps that does not have a monotone path a.s.  As the title of this note suggests, we are particularly interested in the cases where \emph{two parameters} are sufficient to describe the optimal strategy. While one-parameter optimal strategies are ubiquitous,  the study on two-parameter strategies is, to our best knowledge, rather rare.

\subsection{One-parameter strategies} \label{section_one_parameter} In a majority of stochastic control problems that admit analytical solutions, an optimal strategy can typically be described by one parameter.  

In the continuous-time, infinite-horizon \emph{optimal stopping} driven by a one-dimensional Markov process, the stopping and waiting regions are separated by \emph{free boundaries}, and in many cases the boundary is a single point.
 In American/Russian  perpetual options driven by a \lev process, it is known as in, e.g.,  \cite{Avram_2004} and \cite{mordecki2002optimal} that it is optimal to exercise when the process or its reflected process goes above or below a certain barrier for the first time.  In the quickest detection of a Wiener process \cite{shiryaev1961problem} where one wants to detect promptly the unobservable sudden change of the drift of the process, it is optimal to stop when the posterior probability process exceeds some level for the first time.  There are a number of other examples where the first crossing time of a boundary is optimal; see, e.g., \cite{Egami-Yamazaki-2010-1, Kyprianou_Surya_2007, leung_yamazaki_2013}, and also the book by Peskir and Shiryaev \cite{Peskir_Shiryaev_2006}.
 
In singular control, again the controlling and waiting regions are typically separated by a single point. Well-studied examples include  de Finetti's dividend problem, where one wants to maximize the total expected dividends accumulated until ruin (or the first time the [controlled] surplus process goes below zero).  A majority of the existing literature aim to show the optimality of the barrier strategy
that pays dividends  so that the surplus process is reflected at the barrier.  In the spectrally negative \lev model, it has been shown by \cite{Loeffen_2008} that a barrier strategy is optimal on condition that the \lev measure has a completely monotone density.  On the other hand, for the spectrally positive \lev case, optimality is guaranteed as shown in \cite{Bayraktar_2012}. Recently, these results have been extended to the cases when a strategy is assumed to be \emph{absolutely continuous} with respect to the Lebesgue measure: the optimal strategy can again be described by a single threshold, and the so-called \emph{refraction strategy} is optimal; see \cite{kyprianou2012optimal} and \cite{yin2014optimal}.

In the continuous-time inventory model (with the assumption that backorders are allowed), one wants to find an optimal replenishment strategy that minimizes the sum of inventory and controlling costs.  In the spectrally negative \lev case, under e.g.\ the convexity assumption on the inventory cost and with the absence of a fixed cost, it is shown to be optimal to replenish the item so that the inventory does not go below a certain level (see Section 7 of \cite{Yamazaki_2013}).
The absolutely continuous case has been studied by \cite{Hernandez_Perez_Yamazaki_2015} where they showed the optimality of a refraction strategy.


\subsection{Two-parameter strategies} \label{subsection_two_parameter}

In view of the examples above of one-parameter strategies, it is not difficult to see that, by a simple modification to the problem setting, one needs more parameters to describe the optimal strategy.
 Here we list several examples where one additional parameter will also be needed.

\subsubsection{Two-sided singular control} \label{description_two_sided}

In the above examples of singular control, it is assumed that control is one-sided: one can only decrease or increase the underlying process.  However, there are versions where it is two-sided and one can decrease and  also increase the process.

In the extension of de Finetti's problem with \emph{capital injections}, the surplus process can also be increased by injecting capital.  Typically, the problem requires that capital be injected so that the surplus process never goes below zero.
In inventory control, one can think of a version where the item can be replenished and also sold so as to avoid the shortage and excess of an inventory, respectively.

%

\subsubsection{Impulse control} \label{description_impulse_control}

Another extension from singular control can be considered by adding a fixed cost.  Namely, in addition to the cost (or reward) that is proportional to the amount of modification, a fixed cost is incurred each time it is modified.  In this case, it is clear that one parameter is no longer sufficient to describe the optimal strategy.  Instead, one can expect that the $(s,S)$-strategy (more commonly called the $(s,S)$-policy) is a reasonable candidate.  In other words, given two threshold levels $s$ and $S$, whenever the process goes above (or below) $s$, the inventory is pushed down (or up) to $S$.  The optimality of an $(s,S)$-strategy is often a primary objective in the impulse control literature.

%
%

\subsubsection{Zero-sum games between two players}  \label{description_games} In a (stochastic) game, multiple players aim to maximize their own expected payoffs.  However, the payoff depends not only on her action but also on other players' actions. The primary objective of game theory is to identify, if any, a \emph{Nash equilibrium (saddle point)}, which is a set of strategies such that each player cannot increase her expected payoff by solely changing hers, unless other players change their strategies as well.

Consider the case with two players where a common payoff is maximized by one player and is minimized by the other.
 Under the settings similar to those described in Section \ref{section_one_parameter} above, each player's strategy is described by one parameter, and consequently the equilibrium is described by two parameters.

\subsection{Fluctuation theory of spectrally one-sided \lev processes}

In this note, we assume throughout that the underlying (uncontrolled) process is a spectrally negative \lev process. The spectrally positive \lev process is its dual and hence the case driven by this process is also covered.
While spectrally one-sided \lev processes are not necessarily desirable processes for realistic models, at least analytically, it has a great advantage to work with these set of processes.

Over the last decade, significant developments in the fluctuation theory of spectrally one-sided \lev processes have been presented  (see, e.g., the textbooks by Bertoin \cite{Bertoin_1996}, Doney \cite{Doney_2007}, and Kyprianou \cite{Kyprianou_2006}).  Various fluctuation identities are known to be written using the so-called \emph{scale functions}, and these include essentially all the expectations needed to compute the net present values (NPVs) of the payoffs under the one-parameter and two-parameter strategies described above.  

The scale function is defined by its Laplace transform written in terms of the Laplace exponent of the process.  We shall see in this note that, despite its concise characterization, it still contains the information sufficient to solve the problem.

%
%

\subsection{Solution procedures}  

Using the expected NPVs of payoffs under each two-parameter strategy, written explicitly in terms of the scale function, the classical ``\emph{guess and verify}" approach can be carried out in a straightforward manner.  Here, we illustrate each step briefly below.


\subsubsection{Selection of the two parameters} \label{subsection_selection_two_param} As the form of the candidate strategy is already conjectured, the guessing part essentially is to decide on the values of the two parameters.  Because we need to identify two values, naturally we need two equations.

Before discussing on the two-parameter case, let us start with the one-parameter case to gain some intuition.  As reviewed above in Section \ref{section_one_parameter}, the parameter usually corresponds to the value of a barrier. Here, let us temporarily use $u_a(x)$ for the expected NPV when the parameter/barrier is $a$ and the starting value of the process is $x$.

In this case, the most intuitive and straightforward approach is to use the first-order condition. Namely, we first obtain the parameter, say $a^*$, that minimizes or maximizes $a \mapsto u_a(x)$.  Naturally,  it is expected (given that the barrier is in the interior of the state space), the derivative $\partial u_a(x) / \partial a |_{a = a^*}$ must vanish.  This can be easily done because $u_a(x)$ is written using the scale function, whose smoothness is well-studied (see Remark \ref{remark_smoothness} below).

Alternatively, one can apply what is known as \emph{continuous/smooth fit.}  This basically chooses the barrier $a^*$ so that \emph{the degree of smoothness of $u_a(\cdot)$ at $a$ increases by one by setting $a = a^*$}.   The smoothness at the barrier is in general dependent on the regularity (see Section \ref{section_regularity} below for its definition).  In optimal stopping and impulse control, the value function is expected to be  continuous (resp.\ continuously differentiable) at the barrier when it is irregular (resp.\ regular) for the controlling/stopping region.  On the other hand, for singular control, it is expected to be continuously differentiable (resp.\ twice continuously differentiable) at the barrier when it is irregular (resp.\ regular). 

At least for the \lev case, these two methods tend to lead to the same condition, which says that some function, say $a \mapsto g(a)$, of the barrier level $a$ (and not $x$) vanishes; see Figure \ref{one_parameter}.  In addition, under a suitable assumption, it typically is a strictly monotone function. Hence, the candidate barrier can be defined as its unique root.  We refer the reader to \cite{Egami-Yamazaki-2011} for the detailed discussions on the equivalence between these two methods for optimal stopping problems.

 \begin{figure}[htbp]
\begin{center}
\begin{minipage}{1.0\textwidth}
\centering
\begin{tabular}{c}
 \includegraphics[scale=0.4]{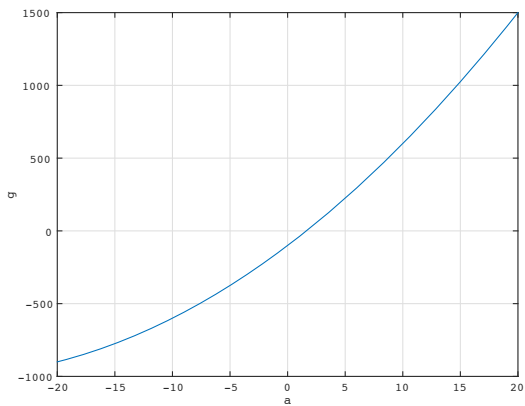}   
\end{tabular}
\end{minipage}
\end{center}
\caption{(One-parameter case) Typical function $a \mapsto g(a)$ obtained when the first-order or continuous/smooth fit condition is applied.  The desired parameter becomes its unique root.}  \label{one_parameter}
\end{figure}

We now move onto the two-parameter case.   Let us temporarily use $v_{a,b}(x)$ for the expected NPV under the strategy parametrized by $(a,b)$ when the starting value of the process is $x$. 

The first approach is again to use the first-order condition. This time, we apply it with respect to the two parameters $(a,b)$, or equivalently we compute the partial derivatives $\partial v_{a,b}(x) / \partial a$ and $\partial v_{a,b}(x) / \partial b$ and choose the parameters so that both of them vanish simultaneously. The second approach is to use  continuous/smooth fit at the barriers (with an additional condition for the case of impulse control).  Again, we end up having the same two equations, say $\Lambda(a,b) = 0$ and $ \lambda(a,b) = 0$.

The difficulty here is that this time we need to show the existence of solutions to the two equations, which are typically nonlinear functions. However, the two equations tend to be related in that one is the partial derivative of the other, i.e., $\lambda(a,b) = \partial \Lambda(a,b) / \partial b$.  In other words, one wants to obtain the curve $b \mapsto \Lambda(a^*, b)$ that touches and gets tangent to the x-axis at $b^*$; see Figure \ref{two_parameter}.

 \begin{figure}[htbp]
\begin{center}
\begin{minipage}{1.0\textwidth}
\centering
\begin{tabular}{c}
  \includegraphics[scale=0.4]{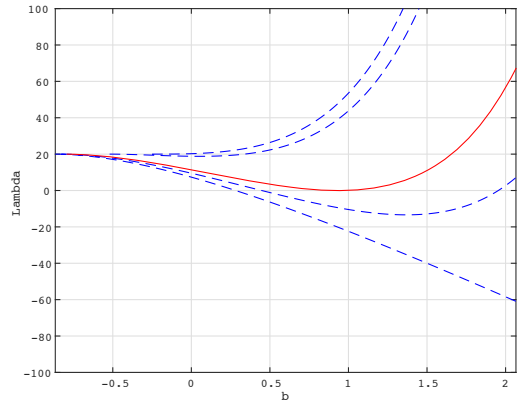}
\end{tabular}
\end{minipage}
\end{center}
\caption{(Two-parameter case) Typical function  obtained when the first-order or continuous/smooth fit condition is applied.  The plot is the curve $b \mapsto \Lambda(a,b)$ on $[a, \infty)$ for different values of $a$.  Typically the desired values $(a^*,b^*)$ become those such that $\lambda(a^*,b^*) = \partial \Lambda(a^*,b) / \partial b|_{b = b^*}$ = 0.  In other words,  one needs to find the starting point $a^*$ such that the curve  gets tangent to the x-axis at $b^*$, as in the solid curve in the plot. }  \label{two_parameter}
\end{figure}


\subsubsection{Verification of optimality}  After the values of the two parameters, say $(a^*,b^*)$, are selected, the optimality of the corresponding strategy must be verified. The so-called \emph{verification lemma} gives a sufficient condition for optimality that commonly require
\begin{enumerate}
\item the smoothness of $v_{a^*,b^*}$,
\item that $v_{a^*,b^*}$ solves the variational inequalities.
\end{enumerate}
The imposed conditions must be sufficient enough so that the discounted process of $v_{a^*,b^*}(\cdot)$ (killed upon exiting the state space), driven by any controlled process, is a local sub/super-martingale.  In general, the forms of the variational inequalities are well-known (see e.g.\ \cite{Oksendal_Sulem_2007}). However, it needs to be customized for technical details, and,  in particular, one needs to take care of the tails of $v_{a^*,b^*}$ and the \lev measure; because of the localizing arguments needed to apply It\^o's formula, one needs, at the end, to take a limit and interchange it over integrals.

Regarding (1), the values of $(a^*,b^*)$ are chosen at the guessing step so that $v_{a^*,b^*}$ is ``sufficiently smooth," although the smoothness at the boundary may not be sufficient enough to apply the usual version of It\^o's formula (and may need the Meyer-It\^o version).  For stochastic calculus for \lev processes, see \cite{MR2273672} and \cite{applebaum2009levy}.

Showing (2) is usually the hardest part, and sometimes it fails.  The variational inequalities need to hold at each point in the state space, which is separated into waiting and controlling regions.  In our examples when the state space is $\R$, except for the impulse control case, the waiting region is given by $(a^*,b^*)$ while the controlling region is $(-\infty, a^*) \cup (b^*, \infty)$.
At a point in the waiting region $(a^*,b^*)$, the proof is normally simple because the discounted process of $v_{a^*,b^*} (\cdot)$ driven by the underlying process is a martingale; see Section \ref{subsection_martigale_properties}.  On the other hand, the proof for the point in $(b^*, \infty)$ (resp.\ $(-\infty, a^*)$) tends to be difficult  for the spectrally negative (resp.\ positive) \lev case.  Intuitively, this is because the process can jump from one region to the other, where the form of $v_{a^*,b^*}$ changes.

\subsection{Comparison with other approaches}

The classical approach for stochastic control for \lev processes involves the integro-differential equations (IDEs). 

The candidate value function is first identified as the solution to an IDE with its boundary conditions given by the desired continuity/smoothness at the barriers. Except for special cases, it cannot be solved analytically, and hence verification arguments must be conducted using this implicit representation of the candidate value function.  This is especially difficult when the \lev measure is an infinite measure.


A clear advantage of using the fluctuation theory approach described above is that, if the function  $v_{a^*,b^*}$  can be computed using the scale function, computation is much more direct and simpler.  While the scale function in general does not admit analytically closed expression, the solution methods do not require details of its form. Typically, the selection of the parameters can be done by its asymptotic property at zero (see Section \ref{subsection_smoothness} below) and, for verification, some general properties of the scale function can be used.


Another advantage is that it can deal with the case with jumps of infinite activity/variation without any additional work.  The IDE approach often needs to assume that the jump part of the underlying process is a compound Poisson process.  However, there are a number of important examples with infinite \lev measures such as variance gamma, CGMY, and normal inverse Gaussian processes as well as classical ones as the gamma process and a subset of stable processes.  
 
%


 
 \subsection{Computation}  Using these approaches, the value function as well as the selected parameters are written in terms of the scale function. Hence the computation of these is essentially equivalent to that of the scale function.  
   Because the scale function is defined by its Laplace transform written in terms of the Laplace exponent,  it needs to be inverted either analytically or numerically.  

Some classes of \lev processes have rational forms of Laplace exponents; for these processes, analytical forms of scale functions can be easily obtained by partial fraction decomposition. Among them, the case with i.i.d.\ phase-type jumps (see \cite{Asmussen_2004}) is particularly important, because at least in principle it can approximate any \lev process.  This means that any scale function can be approximated by the scale function of this process.  Egami and Yamazaki \cite{Egami_Yamazaki_2010_2} conducted a sequence of numerical experiments to confirm the accuracy of this approximation.


Alternatively, the scale function can always be directly computed via numerical Laplace inversion.  As discussed in Kuznetsov et al.\ \cite{Kuznetsov2013}, the scale function can be written as the difference between an exponential function (whose parameter is defined by $\Phi(q)$ in the current note) and the resolvent (potential) term [see the third equation in \eqref{resolvent_density} below].  Hence, the computation is reduced to that of the resolvent term. It is a bounded function that asymptotically converges to  zero, and hence, numerical Laplace inversion can be quickly and accurately conducted. For more details, we refer the readers to Section 5 of \cite{Kuznetsov2013}.

\vspace{0.5cm}


In this note, we give a review on these techniques, using several examples on two-sided singular control, impulse control and games, as reviewed in Section \ref{subsection_two_parameter} above.  It is not our aim to give rigorous arguments and instead we give a guide on how the existing results on the fluctuation theory and scale function can be applied to solve stochastic control problems. For more technical details, we refer the reader to the original works cited throughout the note. 

The rest of the note is organized as follows:

In Section \ref{section_review_levy}, we review the spectrally negative \lev process and the scale function. In particular, we review the fluctuation identities  as well as some important properties of the scale function that will be used later in the note.

In Section \ref{section_singular_control}, we study two-sided singular control as introduced in  Section \ref{description_two_sided}.  We first give the formulation and review several examples.
We then discuss how the two parameters are chosen via continuous/smooth fit  and its optimality is shown via verification arguments.  We, in particular, focus on the problems considered in Bayraktar et al.\ \cite{Bayraktar_2012} and  Baurdoux and Yamazaki \cite{Baurdoux_Yamazaki_2015} and illustrate how these solution procedures can be taken.


In Section \ref{section_impulse_control}, we consider impulse control as addressed in  Section \ref{description_impulse_control}.  We see that the techniques used are similar to those used for singular control.  However, there are several major differences and new challenges in solving. We in particular use the case as in Yamazaki \cite{Yamazaki_2013} to illustrate the steps necessary to solve the problem.

In Section  \ref{section_game}, we study two-player optimal stopping games as in  Section \ref{description_games} with a special focus on the problem studied by Egami et al.\ \cite{Leung_Yamazaki_2011}. 
Some remarks on other forms of two-player zero-sum games are also given.


Throughout this study,  $x+ := \lim_{y \downarrow x}$ and $x-  := \lim_{y \uparrow x}$ are used to indicate the right- and left-hand limits, respectively. 
We let $\Delta \xi_t := \xi_t - \xi_{t-}$, for any right-continuous process $\xi$.  Finally,  for any interval $\mathcal{I} \subset \R$, let $\overline{\mathcal{I}} := \sup \mathcal{I}$, $\underline{\mathcal{I}} := \inf \mathcal{I}$, and $\mathcal{I}^o$ be the interior of $\mathcal{I}$.

\section{Spectrally Negative \lev Processes and Scale Functions} \label{section_review_levy}

In this section, we review the spectrally negative \lev process and its fluctuation theory.  We shall also review the scale function and list the fluctuation identities as well as some important properties that are frequently used in stochastic control.  Note that the spectrally positive \lev process is its dual, and the results introduced here can be directly applied as well.

Defined on a probability space $(\Omega, \mathcal{F}, \p)$, let $X$ be a spectrally negative \lev process with its Laplace exponent $X$ given by
\begin{align}
\psi(s)  := \log \E \left[ e^{s X_1} \right] =  \gamma s +\frac{1}{2}\sigma^2 s^2 + \int_{(-\infty,0)} (e^{s z}-1 - s z 1_{\{x > -1\}} ) \nu (\diff z), \quad s \geq 0, \label{laplace_spectrally_positive}
\end{align}
where $\nu$ is a \lev measure with the support $(-\infty,0)$ that satisfies the integrability condition $\int_{(-\infty,0)} (1 \wedge |z|^2) \nu(\diff z) < \infty$. For every $x \in \R$, let $\p_x$ be the conditional probability under which $X_0 = x$ (in particular, we let $\mathbb{P} \equiv \mathbb{P}_0$), and $\E_x$ and $\E$ be the corresponding expectation operators.  Let $\mathbb{F}$ be the filtration generated by $X$.

The path variation of the process is particularly important in stochastic control, especially when we apply continuous/smooth fit as we shall see in later sections.  For the case of a \lev process, it has paths of \emph{bounded variation} a.s.\ or otherwise it has paths of \emph{unbounded variation} a.s.  The former holds if and only if $\sigma = 0$ and $\int_{(-1,0)}|z| \, \nu(\diff z) < \infty$; in this case, the expression \eqref{laplace_spectrally_positive} can be simplified to
\begin{align*}
\psi(s)   =  \delta s + \int_{(-\infty, 0)} (e^{s z}-1 ) \nu (\diff z), \quad s \geq 0,
\end{align*}
with $\delta := \gamma - \int_{(-1,0)}z\, \nu(\diff z)$.  

Throughout the note, we exclude the case in which $X$ is the negative of a subordinator (i.e., $X$ is monotonically decreasing a.s.). This assumption implies that $\delta > 0$ when $X$ is of bounded variation.  

\subsection{Path variations and regularity} \label{section_regularity}

As defined in Definition 6.4 of \cite{Kyprianou_2006},
we call a point $x$ \emph{regular} for an open or closed set $B$ if $\p_x \{ T_B = 0 \} = 1$ where 
\begin{align*}
T_B:= \inf \{ t > 0: X_t \in B \},
\end{align*}
 and \emph{irregular} if $\p_x \{ T_B = 0 \} = 0$; here and throughout the note, let $\inf \varnothing = \infty$.  By Blumenthal's zero-one law, the probability $\p_x \{ T_B = 0 \}$ is either $0$ or $1$, and hence any point is either regular or irregular.

As summarized in Section 8 of \cite{Kyprianou_2006}, for any spectrally negative \lev process $X$, the point $0$ is regular for $(0, \infty)$, meaning that, if the process starts at $0$, it enters $(0, \infty)$ immediately.  On the other hand, $0$ is regular for $(-\infty, 0)$ if and only if the process has paths of unbounded variation.

We shall see in later sections that the smoothness of the value function at (free) boundaries depends on their regularity.

\subsection{Scale functions}
Fix $q \geq 0$. For any spectrally negative \lev process $X$, its $q$-scale function 
\begin{align*}
W^{(q)}: \R \rightarrow [0,\infty), 
\end{align*}
is a function that is zero on $(-\infty,0)$, continuous and strictly increasing on $[0,\infty)$, and is characterized by the Laplace transform:
\begin{align}
\int_0^\infty e^{-s x} W^{(q)}(x) \diff x = \frac 1
{\psi(s)-q}, \qquad s > \lapinv, \label{scale_function_laplace}
\end{align}
where
\begin{equation}
\lapinv :=\sup\{\lambda \geq 0: \psi(\lambda)=q\}. \notag
\end{equation}
Here, the Laplace exponent $\psi$ in \eqref{laplace_spectrally_positive} is known to be zero at the origin and convex on $[0,\infty)$.   We also define, for $x \in \R$,
\begin{align*}
\overline{W}^{(q)}(x) &:=  \int_0^x W^{(q)}(y) \diff y, \\
Z^{(q)}(x) &:= 1 + q \overline{W}^{(q)}(x),  \\
\overline{Z}^{(q)}(x) &:= \int_0^x Z^{(q)} (z) \diff z = x + q \int_0^x \int_0^z W^{(q)} (w) \diff w \diff z.
\end{align*}
Because $W^{(q)}(x) = 0$ for $-\infty < x < 0$, we have
\begin{align}
\overline{W}^{(q)}(x) = 0, \quad Z^{(q)}(x) = 1  \quad \textrm{and} \quad \overline{Z}^{(q)}(x) = x, \quad x \leq 0.  \label{z_below_zero}
\end{align}
We shall also define, when $\psi'(0+) > -\infty$,
\begin{align*}
R^{(q)} (x) := \overline{Z}^{(q)}(x) + \frac {\psi'(0+)} q, \quad x \in \R.
\end{align*}

In Figure \ref{figure_scale_function}, we show  sample plots of the scale function $W^{(q)}$ on $[0, \infty)$ for the cases of bounded and unbounded variation.  Its behaviors as $x \downarrow 0$ and $x \uparrow \infty$ are reviewed later in this section. 

 \begin{figure}[htbp]
\begin{center}
\begin{minipage}{1.0\textwidth}
\centering
\begin{tabular}{c}
 \includegraphics[scale=0.4]{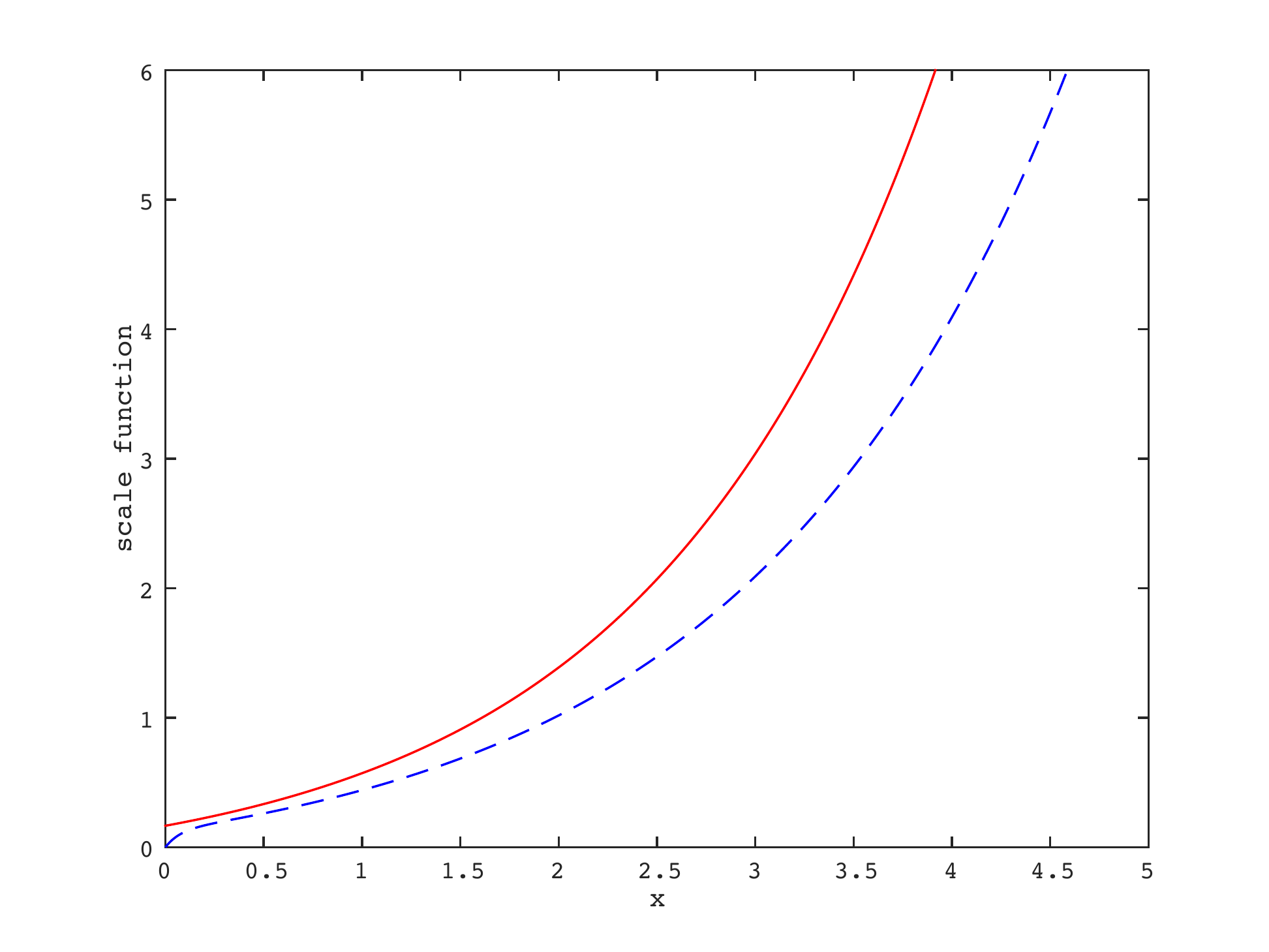}
 \end{tabular}
\end{minipage}
\caption{Plots of the scale function $W^{(q)}$ on $[0, \infty)$. The solid red curve is for the case of bounded variation; the dotted blue curve is for the case of unbounded variation (with $\sigma > 0$). As reviewed in \eqref{eq:Wq0}, its behaviors around zero depend on the path variation of the process.  In addition, as in \eqref{scale_function_asymptotics}, it increases exponentially as $x \rightarrow \infty$.}  \label{figure_scale_function}
\end{center}
\end{figure}

\subsection{Smoothness of scale functions} \label{subsection_smoothness}


A particularly important property of the scale function, which is helpful in applying continuous/smooth fit, is its behaviors around zero: as in Lemmas 3.1 and 3.2 of \cite{Kuznetsov2013}, 
\begin{align}\label{eq:Wq0}
W^{(q)} (0) &= \left\{ \begin{array}{ll} 0, & \textrm{if $X$ is of unbounded
variation,} \\ \frac 1 {\delta}, & \textrm{if $X$ is of bounded variation,}
\end{array} \right. \\
\label{eq:Wqp0}
W^{(q)\prime} (0+) &:= \lim_{x \downarrow 0}W^{(q)\prime} (x) =
\left\{ \begin{array}{ll}  \frac 2 {\sigma^2}, & \textrm{if }\sigma > 0, \\
\infty, & \textrm{if }\sigma = 0 \; \textrm{and} \; \nu(-\infty,0) = \infty, \\
\frac {q + \nu(-\infty, 0)} {\delta^2}, &  \textrm{if }\sigma = 0 \; \textrm{and} \; \nu(-\infty, 0) < \infty.
\end{array} \right.
\end{align}
Note that these can be confirmed in Figure \ref{figure_scale_function}.

As we shall see in later sections, when considering continuity/smoothness at the lower barrier, the difference between the right-hand and left-hand limits often becomes the product of  $W^{(q)}(0)$ and some function, say $\Lambda(a,b)$, of the two parameters (barriers) $(a,b)$ to be selected: for these to match, the parameters $(a,b)$ must be chosen so that either $W^{(q)}(0)$ or $\Lambda(a,b)$ vanishes.  

When $W^{(q)}(0) = 0$ (or equivalently $X$ is of unbounded variation), then the value function is expected to be smoother.  Repeating the same procedure for its derivative, one gets that the difference between the right-hand and left-hand limits becomes the product of $W^{(q) \prime}(0+)$ and $\Lambda(a,b)$; in this case, $(a,b)$ must be chosen so that $\Lambda(a,b) = 0$.

At the upper boundary, the smoothness tends to be the same for both bounded and unbounded variation cases: this gives another equation $\lambda(a,b) = 0$ where $\lambda(a,b)$ is the partial derivative of $\Lambda(a,b)$ with respect to $b$.

Regarding the smoothness of the scale function on $\R \backslash\{0\}$, we have the following; see \cite{Chan_2009} for more comprehensive results.  These smoothness results are important in order to apply It\^o's formula where the  (candidate) value function must be $C^2$ (resp.\ $C^1$) for the case of unbounded (resp.\ bounded) variation.
\begin{remark} \label{remark_smoothness}
If $X$ is of unbounded variation or the \lev measure does not have an atom, then it is known that $W^{(q)}$ is $C^1(\R \backslash \{0\})$.  Hence, 
\begin{enumerate}
\item $Z^{(q)}$ is $C^1 (\R \backslash \{0\})$ and $C^0 (\R)$ for the bounded variation case, while it is $C^2(\R \backslash \{0\})$ and $C^1 (\R)$ for the unbounded variation case,
\item $\overline{Z}^{(q)}$ is $C^2(\R \backslash \{0\})$ and $C^1 (\R)$ for the bounded variation case, while it is $C^3(\R \backslash \{0\})$ and $C^2 (\R)$ for the unbounded variation case.
\end{enumerate}
In addition, if $\sigma > 0$, then $W^{(q)}$ is $C^2(\R \backslash \{0\})$.
\end{remark}

\subsection{Fluctuation identities for spectrally negative \lev processes} \label{fluctuations_underlying}

Here we shall list some fluctuation identities for the spectrally negative \lev process $X$.
\subsubsection{Two-sided exit}
The most well-known application of the scale function is as follows. Let us define the first down- and up-crossing times, respectively, of $X$ by
\begin{align}
\label{first_passage_time}
T_b^- := \inf \left\{ t > 0: X_t < b \right\} \quad \textrm{and} \quad T_b^+ := \inf \left\{ t > 0: X_t >  b \right\}, \quad b \in \R.
\end{align}
Then, for any $b > 0$ and $x \leq b$,
\begin{align}
\begin{split}
\E_x \left[ e^{-q T_b^+} 1_{\left\{ T_b^+ < T_0^- \right\}}\right] &= \frac {W^{(q)}(x)}  {W^{(q)}(b)}, \\
 \E_x \left[ e^{-q T_0^-} 1_{\left\{ T_b^+ > T_0^- \right\}}\right] &= Z^{(q)}(x) -  Z^{(q)}(b) \frac {W^{(q)}(x)}  {W^{(q)}(b)}, \\
 \E_x \left[ e^{-q T_0^-} \right] &= Z^{(q)}(x) -  \frac q {\Phi(q)} W^{(q)}(x).
\end{split}
 \label{laplace_in_terms_of_z}
\end{align}
\subsubsection{Resolvent measures}
The scale function can express concisely the $q$-resolvent  (potential) measure.  As summarized in Theorem 8.7 and Corollaries 8.8 and 8.9 of \cite{Kyprianou_2006} (see also Bertoin \cite{Bertoin_1997}, Emery \cite{Emery_1973}, and Suprun \cite{Suprun_1976}),  we have
\begin{align} \label{resolvent_density}
\begin{split}
\E_x \Big[ \int_0^{T_{0}^- \wedge T^+_b} e^{-qt} 1_{\left\{ X_t \in \diff y \right\}} \diff t\Big] &= \Big[ \frac {W^{(q)}(x) W^{(q)} (b-y)} {W^{(q)}(b)} -W^{(q)} (x-y) \Big] \diff y,  \quad b > 0, \; x \leq b, \\
\E_x \Big[ \int_0^{T_{0}^-} e^{-qt} 1_{\left\{ X_t \in \diff y \right\}} \diff t\Big] &=\left[ e^{- \Phi(q)y} W^{(q)} (x) -W^{(q)} (x-y) \right] \diff y,  \\
\E_x \Big[ \int_0^\infty e^{-qt} 1_{\left\{ X_t \in  \diff y \right\}} \diff t\Big] &=  \left[  \frac {e^{\Phi(q) (x-y)}} {\psi'(\Phi(q))} -W^{(q)} (x-y) \right] \diff y.
\end{split}
\end{align}

Now define, for any measurable function $h$ and $s \in \R$,
\begin{align*} 
\begin{split}
\Psi(s;h) &:= \int_0^\infty e^{- \Phi(q) y}   h(y+s) \diff y = \int_s^\infty e^{- \Phi(q) (y-s)}   h(y) \diff y, \\
\varphi_s (x;h) &:= \int_{s}^x W^{(q)} (x-y) h(y) \diff y, \quad x \in \R.
\end{split}
\end{align*}
Here $\varphi_s (x;h) = 0$ for any $x \leq s$ because $W^{(q)}$ is uniformly zero on $(-\infty,0)$. 
Then it is clear that
\begin{align*} 
\begin{split}
\E_x \Big[ \int_0^{T_{a}^- \wedge T^+_b} e^{-qt} h(X_t) \diff t\Big] &= \frac {W^{(q)}(x-a)} {W^{(q)}(b-a)}\varphi_a(b; h)  - \varphi_a(x; h),  \quad b > a,  \; x \leq b,\\
\E_x \Big[ \int_0^{T_{a}^-} e^{-qt} h(X_t) \diff t\Big] &= \Psi(a;h) W^{(q)}(x-a) - \varphi_a(x; h), \quad x, a \in \R,
\end{split}
\end{align*}
where we assume for the latter that $\Psi(a; h)$ is well-defined and finite.


\subsection{Fluctuation identities for the infimum and reflected processes}

Let us define  the \emph{running infimum and supremum processes} 
\begin{align*}
\underline{X}_t := \inf_{0 \leq t' \leq t} X_{ t'} \quad \textrm{and} \quad \overline{X}_t := \sup_{0 \leq t' \leq t} X_{ t'}, \quad  t \geq 0.
\end{align*}
Then, the processes reflected from above at $b$ and below at $a$ are given, respectively, by
\begin{align*}
\bar{Y}_t^b &:= X_t - D_t^b \quad \textrm{and} \quad \underline{Y}_t^a := X_t + U_t^a, \quad  t \geq 0,
\end{align*}
where
\begin{align*}
D_t^b := (\overline{X}_t -b) \vee 0 \quad \textrm{and} \quad U_t^a := (a - \underline{X}_t) \vee 0, \quad t \geq 0,
\end{align*}
are the cumulative amounts of reflections that push the processes downward and upward, respectively.

\subsubsection{Fluctuation identities for the infimum process}

By Corollary 2.2 of \cite{Kuznetsov2013}, 
\begin{align*}
\E \Big[ \int_0^{\infty} e^{-qt} 1_{\left\{ - \underline{X}_t \in \diff y \right\}} \diff t \Big] = \frac 1 {\Phi(q)} W^{(q)} (\diff y) -  W^{(q)} (y) \diff y = \frac 1 {\Phi(q)}[\Theta^{(q)}(y) \diff y + W^{(q)}(0) \delta_0(\diff y)],  \end{align*}
where $W^{(q)}(\diff y)$ is the measure such that $W^{(q)}(y) = \int_{[0,y]}W^{(q)}(\diff z)$  (see  \cite[(8.20)]{Kyprianou_2006}) and $\delta_0$ is the Dirac measure at zero.  Here, for all $y > 0$,
\begin{align} \label{def_theta}
\begin{split}
\Theta^{(q)}(y) &:= W^{(q)\prime} (y+)- \Phi(q) W^{(q)} (y) 
> 0.
\end{split}
\end{align}
See another probabilistic interpretation of this function in  Section 3.3 in \cite{surya_yamazaki_2014}.  
This function often appears in stochastic control.  See in particular Sections \ref{game_existence_example} and  \ref{section_impulse_control} below and also \cite{surya_yamazaki_2014}.


\subsubsection{Fluctuation identities for $\bar{Y}_t^b$}
Fix $a < b$.
Define  the first down-crossing time of $\bar{Y}_t^b$ as:
\begin{align*}
\overline{\tau}_{a,b} := \inf \{ t > 0: \bar{Y}_t^b < a\}.
\end{align*}
First, the Laplace transform of $\overline{\tau}_{a,b}$ is given, as in  Proposition 2(ii) of \cite{Pistorius_2004}, by
\begin{align*}
\E_x [ e^{- q \overline{\tau}_{a,b}} ] =Z^{(q)}(x-a) - q W^{(q)}(b-a)  \frac {W^{(q)}(x-a)} {W^{(q)\prime}((b-a)+)}, \quad x \leq b.
\end{align*}
Second, using its resolvent  given in Theorem 1(ii) of \cite{Pistorius_2004}, we have, for $x \leq b$,
\begin{align*}
\E_x \Big[ \int_0^{\overline{\tau}_{a,b}} e^{-qt}  h (\bar{Y}_t^b) \diff t   \Big] 
& = \frac {W^{(q)}(x-a)} {W^{(q) \prime}((b-a)+)}  \left[  W^{(q)}(0) h(b) + \int_a^b h(y) W^{(q) \prime}(b-y)  \diff y \right] - \varphi_a(x; h).
\end{align*}
Finally, as in Proposition 1 of \cite{Avram_et_al_2007}, the discounted cumulative amount of reflection from above is given by
\begin{align*}
\E_x \Big[ \int_{[0,\overline{\tau}_{a,b}]} e^{-qt} \diff D_t^b \Big] = \frac {W^{(q)}(x-a)} {W^{(q)\prime}((b-a)+)}, \quad x \leq b.
\end{align*}

\subsubsection{Fluctuation identities for $\underline{Y}_t^a$}
Fix $a < b$. Define  the first up-crossing time of $\underline{Y}_t^a$ as:
\begin{align*}
\underline{\tau}_{a,b} := \inf \{ t > 0: \underline{Y}_t^a > b\}.
\end{align*}
First, as in page 228 of \cite{Kyprianou_2006}, its Laplace transform is concisely given by
\begin{align*}
\E_x [ e^{- q \underline{\tau}_{a,b}} ] = \frac {Z^{(q)}(x-a)} {Z^{(q)}(b-a)}, \quad x \leq b.
\end{align*}
Second, by Theorem 1(i) of \cite{Pistorius_2004}, for any $x \leq b$,
\begin{align*}
\E_x \Big[ \int_0^{\underline{\tau}_{a,b}} e^{-qt}  h (\underline{Y}_t^a) \diff t   \Big] 
&=  \frac {Z^{(q)}(x-a)} {Z^{(q)}(b-a)}\varphi_a (b;h) -  \varphi_a (x;h).
\end{align*}
Finally, as in the proof of Theorem 1 of \cite{Avram_et_al_2007}, the discounted cumulative amount of reflection from below, given $\psi'(0+) > -\infty$, is 
\begin{align*}
\E_x \Big[ \int_0^{{\underline{\tau}_{a,b}}} e^{-qt} \diff  U_t^a \Big] &= - R^{(q)}(x-a)  + Z^{(q)}(x-a) \frac {R^{(q)}(b-a) } {Z^{(q)}(b-a)}, \quad  x \leq b.
\end{align*}

\subsection{Fluctuation identities for doubly reflected \lev processes} Fix $a < b$.
As a variant of the reflected processes addressed above, the
 \emph{doubly reflected \lev process} is given by 
\begin{align}
Y_t^{a,b} := X_t + U_t^{a,b} - D_t^{a,b}, \quad t \geq 0. \label{doubly_reflected_def}
\end{align}
This process is reflected at the two barriers $a$ and $b$ so as to stay on the interval $[a,b]$;  see page 165 of \cite{Avram_et_al_2007} for the construction of the processes $U^{a,b}$, $D^{a,b}$, and $Y^{a,b}$.  To put it simply, $U^{a,b}$ is activated whenever $Y^{a,b}$ attempts to downcross $a$ so that it stays at or above $a$; similarly, $D^{a,b}$ is activated so that $Y^{a,b}$ stays at or below $b$. 

First, as in Theorem 1 of \cite{Avram_et_al_2007}, for $x \leq b$,
\begin{align} \label{reflection_double_reflected}
\begin{split}
\E_x \left[ \int_{[0,\infty)} e^{-qt} \diff D_t^{a,b} \right] &= \frac {Z^{(q)}(x-a)} {q W^{(q)}(b-a)}, \\
\E_x \left[ \int_{[0,\infty)} e^{-qt} \diff U_t^{a,b} \right] &= - R^{(q)}(x-a)  + \frac {Z^{(q)}(b-a)} {q W^{(q)}(b-a)} Z^{(q)}(x-a),
\end{split}
\end{align}
where we assume $\psi'(0+) > -\infty$ for the latter.

Second, using the $q$-resolvent density of $Y^{a,b}$ given in Theorem 1 of \cite{Pistorius_2003}, we have, for $x \leq b$,
\begin{align} \label{resolvent_doubly_reflected}
\E_x \left[ \int_{[0,\infty)} e^{-qt} h(Y_t^{a, b}) \diff t \right] &= \int_a^b h(y) \left[ \frac {Z^{(q)}(x-a) W^{(q)\prime}(b-y)} {q W^{(q)}(b-a)} - W^{(q)}(x-y) \right] \diff y \\ &+  h(b) \Big[ Z^{(q)}(x-a) \frac {W^{(q)}(0)} {q W^{(q)}(b-a)} \Big].
\end{align}



\subsection{Other properties of the scale function}  Here we list some other properties of the scale function that are often useful in solving stochastic control problems.

\subsubsection{Asymptotics as $x \rightarrow \infty$} \label{subsection_asymptotics} Suppose $q > 0$.  It is known that the scale function $W^{(q)}$ increases exponentially: we have
\begin{align}
 W^{(q)} (x) / e^{\Phi(q)x} \xrightarrow{x \rightarrow \infty} \psi'(\Phi(q))^{-1}. \label{scale_function_asymptotics}
\end{align}
By this, the following limits are also immediate:
\begin{align*}
\lim_{x \rightarrow \infty}\frac {W^{(q) \prime}(x+)} {W^{(q)}(x)} = \Phi(q), \quad \lim_{x \rightarrow \infty}\frac {Z^{(q)}(x)} {W^{(q)}(x)} = \frac q {\Phi(q)} \quad \textrm{and} \quad \lim_{x \rightarrow \infty}\frac {\overline{Z}^{(q)}(x)} {W^{(q)}(x)} = \frac q {\Phi^2(q)}.
\end{align*}
Note also that, for $s \in \R$ and any measurable function $h$ such that $\Psi(s; h)$ is well-defined,
\begin{align}
\lim_{x \rightarrow \infty}\frac {\varphi_s (x;h)} {W^{(q)}(x-s)} = \Psi(s;h). \label{conv_phi_psi}
\end{align}

\subsubsection{Log-concavity}  \label{section_log_concavity} The scale function $W^{(q)}$ is known to be log-concave: as in (8.18) and Lemma 8.2 of \cite{Kyprianou_2006},
\begin{align*}
 \frac {W^{(q)\prime}(y+)} {W^{(q)}(y)} \leq \frac {W^{(q)\prime}(x+)} {W^{(q)}(x)},  \quad  y > x > 0.
\end{align*}
In addition, $W^{(q)\prime}(x-) \geq W^{(q)\prime}(x+)$ for all $x > 0$.  These properties are sometimes needed for the monotonicity of related functions; see Sections \ref{remark_other_impulse} and \ref{game_verification_example} below.

\subsubsection{Martingale properties} \label{subsection_martigale_properties}

Let $\mathcal{L}$ be the infinitesimal generator associated with
the process $X$ applied to a \emph{sufficiently smooth} function $h$  (i.e.\ $C^1$ [resp.\ $C^2$] for the case $X$ is of bounded [resp.\ unbounded] variation): for $x \in \R$,
\begin{align} \label{generator}
\begin{split}
\mathcal{L} h(x) &:= \gamma h'(x) + \frac 1 2 \sigma^2 h''(x) + \int_{(-\infty,0)} \left[ h(x+z) - h(x) -  h'(x) z 1_{\{-1 < z < 0\}} \right] \nu(\diff z), \\
\textrm{(resp. }\mathcal{L} h(x) &:= \delta h'(x) +  \int_{(-\infty,0)} \left[ h(x+z) - h(x)  \right] \nu(\diff z)\textrm{).}
\end{split}
\end{align}
The variational inequalities are written using this generator with $h$ replaced with the candidate value function.  Typically, it makes sense (except at the selected [free] boundaries), thanks to its smoothness that can be confirmed by that of the scale function as in Remark \ref{remark_smoothness}.  At the boundaries,  for optimal stopping and impulse control, the function may not be smooth enough and hence \eqref{generator} is not well-defined, although its right and left limits normally exist and are finite. In such cases, the Meyer-It\^o formula (see, e.g., Theorem 71 of Protter \cite{MR2273672}) is used in the proof of verification lemma.

One useful known fact regarding the generator \eqref{generator} is as follows. By Proposition 2 of \cite{Avram_et_al_2007} and as in the proof of Theorem 8.10 of \cite{Kyprianou_2006}, the processes 
\begin{align*}
e^{-q (t \wedge T^-_{0} \wedge T^+_B)} Z^{(q)}(X_{t \wedge T^-_{0} \wedge T^+_B}) \quad \textrm{and} \quad e^{-q (t \wedge T^-_{0} \wedge T^+_B)}  R^{(q)} (X_{t \wedge T^-_{0} \wedge T^+_B}), \quad t \geq 0,
\end{align*}
for any $B > 0$ are martingales, where we assume $\psi'(0+) > -\infty$ for the latter.
Thanks to the smoothness of $Z^{(q)}$ and $\overline{Z}^{(q)}$ on $(0,\infty)$ as in Remark \ref{remark_smoothness}, we obtain 
\begin{align}
(\mathcal{L}-q) Z^{(q)}(y) =(\mathcal{L}-q) R^{(q)}(y) = 0, \quad y > 0. \label{martingale_Z_R}
\end{align}
The same result holds for $W^{(q)}$ and 
\begin{align}
(\mathcal{L}-q) W^{(q)}(y) = 0, \quad y > 0,  \label{martingale_W}
\end{align}
 on condition that it is sufficiently smooth.

Another useful known fact is that, as in the proof of Lemma 4.5 of \cite{Egami-Yamazaki-2010-1}, if $h$ is continuous,
\begin{align}
(\mathcal{L}-q) \varphi_{s} (x;h) = h(x), \quad x > s.
\end{align}
These properties are often sufficient to prove that the candidate value function is harmonic in the waiting (non-controlling) region.

\subsection{Some further notations} Before closing this section, we shall define, if they exist, the following threshold levels.

\begin{definition} \label{def_a_bar}
Given a closed interval $\mathcal{I} \subset \R$ and a measurable function $h$, let  $\bar{a} = \bar{a}(h) \in \mathcal{I}$ be such that $h(x) < 0$
for  $x \in (-\infty, \bar{a}) \cap \mathcal{I}$, and $h(x) > 0$ for $x \in (\bar{a},\infty) \cap \mathcal{I}$, if such a value exists.  If $h(x) < 0$ for $x \in \mathcal{I}$, then we set $\bar{a} = \bar{a}(h) = \overline{\mathcal{I}}$.   If $h(x) > 0$ for $x \in \mathcal{I}$, then we set $\bar{a} = \bar{a}(h) = \underline{\mathcal{I}}$.
\end{definition}

\begin{definition} \label{def_a_bar_under}
Given a closed interval $\mathcal{I} \subset \R$ and a measurable function $h$ such that $\Psi(x; h)$ is well-defined and finite for all $x \in \mathcal{I}$, let  $\underline{a} = \underline{a}(h) \in \mathcal{I}$ be such that  $\Psi(x;h) < 0$
for  $x \in (-\infty, \underline{a}) \cap \mathcal{I}$, and $\Psi(x;h) > 0$ for $x \in (\underline{a},\infty) \cap \mathcal{I}$, if such a value exists.   If $\Psi(x; h) < 0$ for $x \in \mathcal{I}$, then we set $\underline{a} = \underline{a}(h) = \overline{\mathcal{I}}$.   If $\Psi(x; h) > 0$ for $x \in \mathcal{I}$, then we set $\underline{a} = \underline{a}(h) = \underline{\mathcal{I}}$.
\end{definition}

These values  for a suitably chosen (often monotone) function $h$ give us particularly important information.  Typically, as in the examples shown in later sections, the values of $\underline{a}$ and $\bar{a}$ can act as upper or lower bounds of the two parameters $(a^*, b^*)$ to be chosen.  See, in particular,  Sections \ref{two_sided_control_existence_example}, \ref{inventory_existence_example} and \ref{game_existence_example} and also Tables \ref{table_two_sided}, \ref{table_impulse}, \ref{table_game}.

In addition, the value $\underline{a}$ can be understood as the optimal parameter $a^*$ when the other parameter is $b^* = \infty$. We will also see that the value $\bar{a}$ is important in the verification step; see Lemmas \ref{lemma_easy_results_doubly_reflected}(2), \ref{lemma_easy_results_impulse}(2), and \ref{lemma_verification_games}(2).

\section{Two-sided Singular Control} \label{section_singular_control}
In this section, we consider the singular control problem where one can increase and also decrease the underlying process.
An admissible strategy $\pi := \left\{ (U_t^{\pi}, D_t^{\pi}); t \geq 0 \right\}$ is given by a pair of nondecreasing, right-continuous, and $\mathbb{F}$-adapted processes with $U^\pi_{0-}=D^\pi_{0-}=0$ such that the controlled process
 \begin{align*}
 Y_{t}^\pi := X_t + U_t^\pi - D_t^\pi, \quad t \geq 0,
 \end{align*}
stays in some given closed interval $\mathcal{I}$ uniformly in time. 
Let $\Pi$ be the set of all admissible strategies.

We consider the sum of the running and controlling costs; its expected NPV is given by
\begin{align*}
v^\pi (x) := \E_x \Big[ \int_0^\infty e^{-qt}  f (Y_{t}^\pi) \diff t  + \int_{[0,\infty)}e^{-qt}\left(C_U \diff U^\pi_{t} + C_D \diff D^\pi_{t} \right) \Big], \quad x \in \R,
\end{align*}
for $q > 0$, some continuous and piecewise continuously differentiable function $f$ on $\mathcal{I}$ and fixed constants $C_U, C_D \in \R$ satisfying \begin{align}
C_U + C_D > 0. \label{assump_C_sum}
\end{align}
Here, if $x < \underline{\mathcal{I}}$ (resp.\ $x > \overline{\mathcal{I}}$), then $U_0^\pi = \Delta U_0^\pi  = \underline{\mathcal{I}} - x$ (resp.\ $D_0^\pi = \Delta D_0^\pi  = x - \overline{\mathcal{I}}$) so that $Y_0^\pi \in \mathcal{I}$.

The problem is to compute the value function given by
\begin{align*}
v(x) := \inf_{\pi \in \Pi} v^\pi(x), \quad x \in \R,
\end{align*}
 and the optimal strategy that attains it, if such a strategy exists. 
 
 Throughout this and next sections, let us also use the slope-changed version of $f$ given by
 \begin{align}
\tilde{f}(x) &:= f(x) + C_U q x, \quad x \in \R. \label{def_f_tilde}
\end{align}
The roles and significance of this function will be clear shortly.  We also assume the following so that the expected NPV associated with $U_t^\pi$ is finite.
\begin{assump} \label{finiteness_X_1}
We assume $\E X_1 = \psi'(0+) > -\infty$.
\end{assump}

\begin{example} \label{example_capital_injection}In the optimal dividend problem with capital injections driven by a spectrally negative \lev process, it is required that the controlled risk process stay nonnegative uniformly in time (i.e.\ $\mathcal{I} = [0, \infty)$).  One wants to maximize the expected NPV of discounted dividends minus that for capital injections. This is a maximization problem with $U_t^\pi$ and $D_t^\pi$ being, respectively, the cumulative amounts of capital injections and dividends until $t \geq 0$.  We can formulate this as a minimization problem as above by setting $C_D = -1$ and $C_U = \beta$ where $\beta > 1$ is the unit cost of capital injection.  Here $f$ is assumed to be zero.  This problem has been solved by Avram et al.\  \cite{Avram_et_al_2007} for a general spectrally negative \lev process.
\end{example}
\begin{example} \label{example_capital_injection_dual} In the dual model of Example \ref{example_capital_injection}, it is assumed that the underlying process is a spectrally positive \lev process.
By flipping the processes  with respect to the origin, it is easy to see that the problem is equivalent to the above formulation driven by a spectrally negative \lev process with $\mathcal{I} = (-\infty, 0]$, $C_D = \beta$ and $C_U = -1$. This problem has been solved by  Bayraktar et al.\ \cite{Bayraktar_2012} for a general spectrally positive \lev process.
%
%
\end{example}

\begin{example} \label{example_two_sided_control} A version of continuous-time inventory control considers the case where inventory can be increased (replenished) and decreased (sold).  With the absence of fixed costs and if backorders are allowed, the problem can be formulated as above with $\mathcal{I} = \R$.
 In currency rate control (see, e.g., \cite{jeanblanc1993impulse, mundaca1998optimal}), where a central bank controls the currency rate so as to prevent it from going too high or too low, can also be modeled in the same way.
The classical Brownian motion and continuous diffusion models have been solved by \cite{MR716123} and \cite{matomaki2012solvability}, respectively.  In Baurdoux and Yamazaki \cite{Baurdoux_Yamazaki_2015}, it has been solved for a general spectrally negative \lev process.  In this note, we assume that $f$ is convex for this example.


\end{example}

\subsection{The double reflection strategy} 

In all the examples above, the optimal strategy is shown to be a \emph{double barrier strategy} $\pi_{a,b} := \{ U^{a,b}, D^{a,b}\}$ with the resulting controlled process being the
 doubly reflected \lev process given in \eqref{doubly_reflected_def}.

By \eqref{reflection_double_reflected} and \eqref{resolvent_doubly_reflected}, we can directly compute, for $a < b$,
\begin{align*}
v_{a,b} (x) &:= \E_x \Big[ \int_0^\infty e^{-qt}  f (Y_{t}^{a,b}) \diff t  + \int_{[0,\infty)}e^{-qt} (C_U \diff U^{a,b}_{t} + C_D \diff D^{a,b}_{t} ) \Big], \quad x \in \R.
\end{align*}
For $x \leq b$, it is given by
\begin{align} \label{expression_v_a_b}
\begin{split}
v_{a,b} (x) &=\frac  {\Lambda(a,b)} {q W^{(q)}(b-a)}  {Z^{(q)}(x-a)} -C_U R^{(q)}(x-a) + \frac {f(a)} q Z^{(q)} (x-a) - \varphi_{a}(x; f) \\
\end{split}
\end{align}
where
\begin{align} \label{about_gamma}
\begin{split}
\Lambda(a,b) 
&:= C_D + C_U + \varphi_a (b;\tilde{f}'), \quad b \geq a.
\end{split}
\end{align}
For $x > b$, we have $v_{a,b} (x) = v_{a,b} (b) + C_D (x-b)$.

\begin{remark} In particular, when $f \equiv 0$ (as in Examples \ref{example_capital_injection} and \ref{example_capital_injection_dual} above), for $a < b$,
\begin{align*}
\Lambda(a,b) &= C_D + C_U Z^{(q)}(b-a), \\
v_{a,b} (x) 
&= \frac {C_D + C_U Z^{(q)}(b-a) } {q W^{(q)}(b-a)}{Z^{(q)}(x-a)} -C_U R^{(q)}(x-a), \quad x \leq b;
\end{align*}
see \cite{Avram_et_al_2007}  and \cite{Bayraktar_2012}.
\end{remark}

\subsection{Smoothness of the value function} \label{smoothness_double_barrier}

Focusing on the set of double barrier strategies, the first step is to narrow down to a candidate optimal strategy by deciding on the threshold values, say $a^*$ and $b^*$.  Because the spectrally negative \lev process can reach any point with positive probability, we must have that  $[a^*, b^*] \subset \mathcal{I}$.

As we have discussed in Section \ref{subsection_selection_two_param}, the two parameters can be  identified by the first-order condition or the smooth fit condition.  The first approach uses the first-order conditions at $a^*$ and $b^*$; because $a^*$ and $b^*$ must minimize $v_{a,b}$ over $a$ and $b$,  partial derivatives $\partial v_{a,b} (x)/ \partial a |_{a = a^*, b = b^*}$ and $\partial v_{a,b} (x)/ \partial b |_{a = a^*, b = b^*}$ must vanish, at least when the minimizers are in the interior of $\mathcal{I}$.  The second approach uses the condition that the value function is smooth.  Here, we focus on the second smoothness approach because the computation is slightly easier, and  we need to confirm  the smoothness of $v_{a^*, b^*}$ after all when we verify its optimality.


In singular control, the value function normally admits twice continuous differentiability (resp.\ continuous differentiability)  at each interior point in $\mathcal{I}$ when it is regular (resp.\ irregular).  Thanks to the smoothness of the scale function as in Remark \ref{remark_smoothness},  the only points of $v_{a^*,b^*}$ we need to pay attention are $a^*$ and $b^*$ where the functions are pasted together.
Due to the asymmetry of the spectrally negative \lev process, what we observe at these two points will be different.  Here, recall the definition of regularity and its relation  with the path variation of the process as reviewed in Section \ref{section_regularity}.

Regarding the smoothness of the value function at the lower barrier $a^*$,
\begin{enumerate}
\item if $a^*$ is regular for $(-\infty, a^*)$ (or equivalently $X$ is of unbounded variation), then the twice continuous differentiability at $a^*$ is expected;
\item if $a^*$ is irregular for $(-\infty, a^*)$ (or equivalently $X$ is of bounded variation), then the continuous differentiability at $a^*$ is expected.
\end{enumerate}
Regarding the smoothness at the upper barrier $b^*$, because it is always regular for $(b^*, \infty)$, twice-differentiability is expected at $b^*$ regardless of the path variation of $X$.

These procedures can be carried out in a straightforward fashion by using the expression \eqref{expression_v_a_b} in terms of the scale function.
By taking derivatives in \eqref{expression_v_a_b} and using \eqref{def_f_tilde},
\begin{align} \label{v_derivative_general}
\begin{split}
v_{a,b}' (x)&= \frac { \Lambda(a,b)} {W^{(q)}(b-a)} {W^{(q)}(x-a)} -C_U  -\varphi_a (x;\tilde{f}'), \quad a < x < b, \\
v_{a,b}'' (x+)&= \frac {\Lambda(a,b)} {W^{(q)}(b-a)} W^{(q)\prime}((x-a)+)  -\int_a^x W^{(q)\prime}(x-y) \tilde{f}'(y) \diff y -  \tilde{f}'(x+) W^{(q)}(0), \quad a < x < b.
\end{split}
\end{align}
In view of the former of \eqref{v_derivative_general},  by \eqref{about_gamma},
\begin{align}\label{smooth_fit_all_b}
\begin{split}
v_{a,b}' (b-)&= C_D = v_{a,b}'(b+), \\
v_{a,b}' (a+)&=   \frac {\Lambda(a,b)} {W^{(q)}(b-a)}  {W^{(q)}(0)} -C_U = \frac {\Lambda(a,b)} {W^{(q)}(b-a)}  {W^{(q)}(0)} + v'_{a,b} (a-). 
\end{split}
\end{align}
In other words, the continuous differentiability of $v_{a,b}$ holds at $b$ regardless of the path variation.  On the other hand, in view of \eqref{eq:Wqp0}, while the differentiability at $a$ holds for the case of unbounded variation, it only holds if
\begin{align}
\mathfrak{C}_a: \frac {\Lambda(a,b)} {W^{(q)}(b-a)} = 0 \label{smoothness_condition1}
\end{align}
for the case of bounded variation. Here, the case $b = \infty$ is understood as $\lim_{b \rightarrow \infty}\Lambda(a,b)/W^{(q)}(b-a) = 0$ where by \eqref{conv_phi_psi} we can show that 
\begin{align}
\lim_{b \rightarrow \infty} \frac {\Lambda(a,b)} {W^{(q)}(b-a)} = \Psi(a; \tilde{f}'). \label{limit_Gamma_a_b}
\end{align}

In view of the latter of \eqref{v_derivative_general},
\begin{align*}
v_{a,b}'' (b-) &= \frac {\Lambda(a,b) } {W^{(q)}(b-a)} {W^{(q)\prime}((b-a)-)}  -\lambda(a,b), \\
v_{a,b}'' (a+)&= \frac {\Lambda(a,b) }  {W^{(q)}(b-a)}  {W^{(q)\prime}(0+)} - \tilde{f}'(a+) W^{(q)}(0),
\end{align*}
where
\begin{align}
\lambda(a,b) :=  \frac \partial {\partial b} \Lambda(a,b-) =  \int_a^b W^{(q)\prime}(b-y) \tilde{f}'(y) \diff y +  \tilde{f}'(b-) W^{(q)}(0), \quad b > a. \label{small_gamma}
\end{align}
For the unbounded variation case where the continuous differentiability at $a$ automatically holds, again by 
\eqref{eq:Wqp0}, its twice continuous differentiability holds on condition that $\mathfrak{C}_a$ holds.  Now, for both the bounded and unbounded variation cases, the twice continuous differentiability at $b$ holds if
\begin{align}
\mathfrak{C}_b: \frac {\Lambda(a,b) } {W^{(q)}(b-a)} {W^{(q)\prime}((b-a)-)}  -\lambda(a,b) = 0. \label{smoothness_condition2}
\end{align}  
In particular, \emph{on condition that $\mathfrak{C}_a$ holds}, the condition $\mathfrak{C}_b$ can be simplified to 
\begin{align}
\mathfrak{C}_b': \lambda(a,b) = 0. \label{smoothness_condition2_prime}
\end{align}  

\begin{remark}
When $f \equiv 0$, the conditions $\mathfrak{C}_a$ and $\mathfrak{C}_b$, respectively, are simplified to 
\begin{align}
\mathfrak{C}_a^0&: \frac {C_D + C_U Z^{(q)}(b-a)} {W^{(q)}(b-a)} = 0, \label{smoothness_condition1_f0} \\
\mathfrak{C}_b^0&: \frac {C_D + C_U Z^{(q)}(b-a)} {W^{(q)}(b-a)} {W^{(q)\prime}((b-a)-)}  - q C_U W^{(q)}(b-a) = 0. \label{smoothness_condition2_f0}
\end{align}
\end{remark}

These conditions on $a$ and $b$ can be used to identify the pairs $(a^*, b^*)$.  However, these do not necessarily hold unless $a^*, b^* \in int(\mathcal{I})$.  Here, we give examples where $a^*$ and/or $b^*$ become boundaries of $\mathcal{I}$.
\begin{remark} \label{remark_a_b_single}
\begin{enumerate}
\item In Example \ref{example_capital_injection}, it is expected, because $\beta > 1$ (the unit cost of capital injection is higher than the unit reward of dividend), that capital is injected only when it is necessary to make the company alive, and hence $a^* = 0$.  
\item Similarly, under the formulation with the underlying spectrally negative \lev process described in Example \ref{example_capital_injection_dual},  it is expected that $b^* = 0$.
\item In Example \ref{example_two_sided_control}, if the increment of $f$ as $|x| \rightarrow \infty$ is at most linear and small in comparison to the unit controlling costs $C_U$ and $C_D$, it may not be desirable to activate at all the processes $U^\pi$ and/or $D^\pi$. Hence, $a^* = -\infty$ and/or $b^* = \infty$. 
\end{enumerate}
\end{remark}

%
%


\subsection{Existence of $(a^*, b^*)$}  \label{subsection_existence_a_b}

The first challenge is to show the existence of such $(a^*, b^*)$.  Here, we assume the following.

\begin{assump} \label{assump_f_tilde} 
We assume that $\bar{a} \equiv \bar{a}(\tilde{f}')$ (see Definition \ref{def_a_bar}) exists and is finite, where $\tilde{f}'$ is understood as its right-hand derivative if not differentiable.
\end{assump}
We shall see that $\bar{a}$ is a point such that $a^*$ lies on the left of $\bar{a}$ and $b^*$ lies on its right; see Table \ref{table_two_sided}.

\subsubsection{The case of Example \ref{example_capital_injection}} \label{subsection_example_capital_injection}
 It is clear that Assumption \ref{assump_f_tilde} is satisfied with $\bar{a} = 0$.  As in Remark \ref{remark_a_b_single}(1), $a^* = 0 = \bar{a} = \underline{\mathcal{I}}$.  
Therefore, the condition $\mathfrak{C}_a^0$ has no effect and we only require $\mathfrak{C}_b^0$ which reduces to
\begin{align}
\frac {C_D + C_U Z^{(q)}(b)} {W^{(q)}(b)} {W^{(q)\prime}(b-)}  - q C_U W^{(q)}(b) = 0. \label{gamma_condition_SN_dividend}
\end{align}
Hence, $b^* > 0 = \bar{a} = \overline{\mathcal{I}}$ can be chosen as the smallest value of $b$ such that \eqref{gamma_condition_SN_dividend} holds.
This matches the condition given in (5.6) of \cite{Avram_et_al_2007}.

\subsubsection{The case of Example \ref{example_capital_injection_dual}} \label{subsection_example_capital_injection_dual}
 Again, Assumption \ref{assump_f_tilde} is satisfied with $\bar{a} = 0$.
 Because $C_D = \beta$ and $C_U = -1$, there is a unique $a^* < 0 = \bar{a}$ that satisfies $\mathfrak{C}_a^0$ or equivalently that
 \begin{align}
C_D + C_U Z^{(q)}(-a^*) = 0. \label{gamma_condition_SP_dividend}
\end{align}
Hence, the candidate optimal strategy is given by $a^* =-(Z^{(q)})^{-1} (- C_D / C_U)=-(Z^{(q)})^{-1} (\beta)$ and $b^* = 0$.  This matches the result in \cite{Bayraktar_2012}.  
 

\subsubsection{The case of Example \ref{example_two_sided_control}}  \label{two_sided_control_existence_example} For Example \ref{example_two_sided_control}, 
we  want  a pair $(a^*, b^*)$ such that \eqref{smoothness_condition1} and \eqref{smoothness_condition2} hold simultaneously. Equivalently, we want $(a^*, b^*)$ such that the function $b \mapsto \Lambda(a^*, b)$ attains a minimum $0$ at $b^*$ (if $b^* < \infty$).  Note that, for any $a \in \R$,  $b \mapsto \Lambda(a,b)$ starts at $\Lambda(a,a) = C_D + C_U > 0$.

In this case, $\bar{a}$ always exists by the assumption that $f$ is convex.  Assumption \ref{assump_f_tilde} requires that it is finite.  Recall now Definition \ref{def_a_bar_under}.  The convexity assumption and Assumption \ref{assump_f_tilde} guarantees that $\underline{a} =\underline{a}(\tilde{f}')$ also exists and is finite (with the understanding that $\tilde{f}'$ is the right-hand derivative if it is not differentiable). Note that necessarily $\underline{a} < \overline{a}$.

%
%
%
%
%


 \begin{figure}[htbp]
\begin{center}
\begin{minipage}{1.0\textwidth}
\centering
\begin{tabular}{cc}
 \includegraphics[scale=0.4]{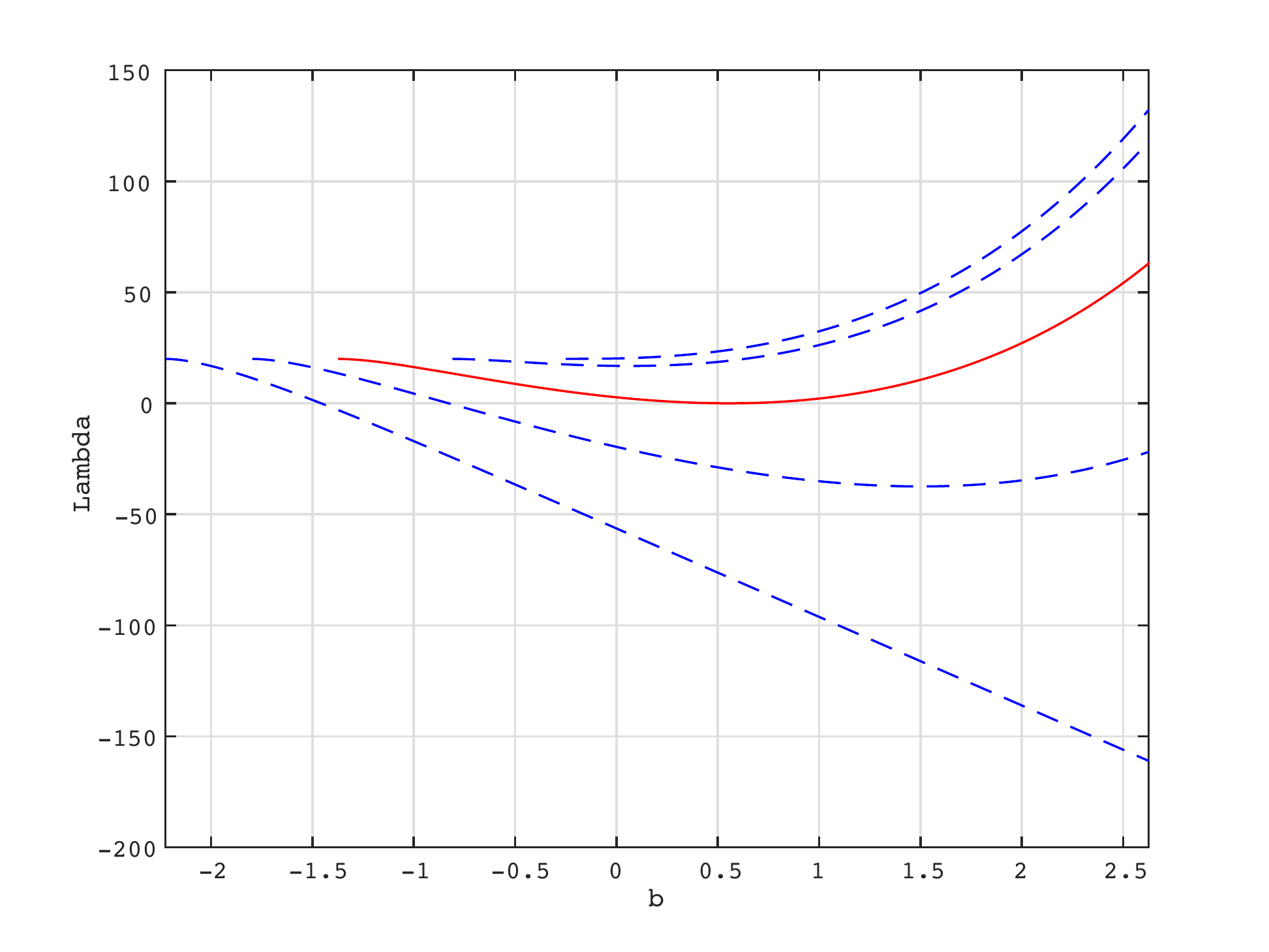} & \includegraphics[scale=0.4]{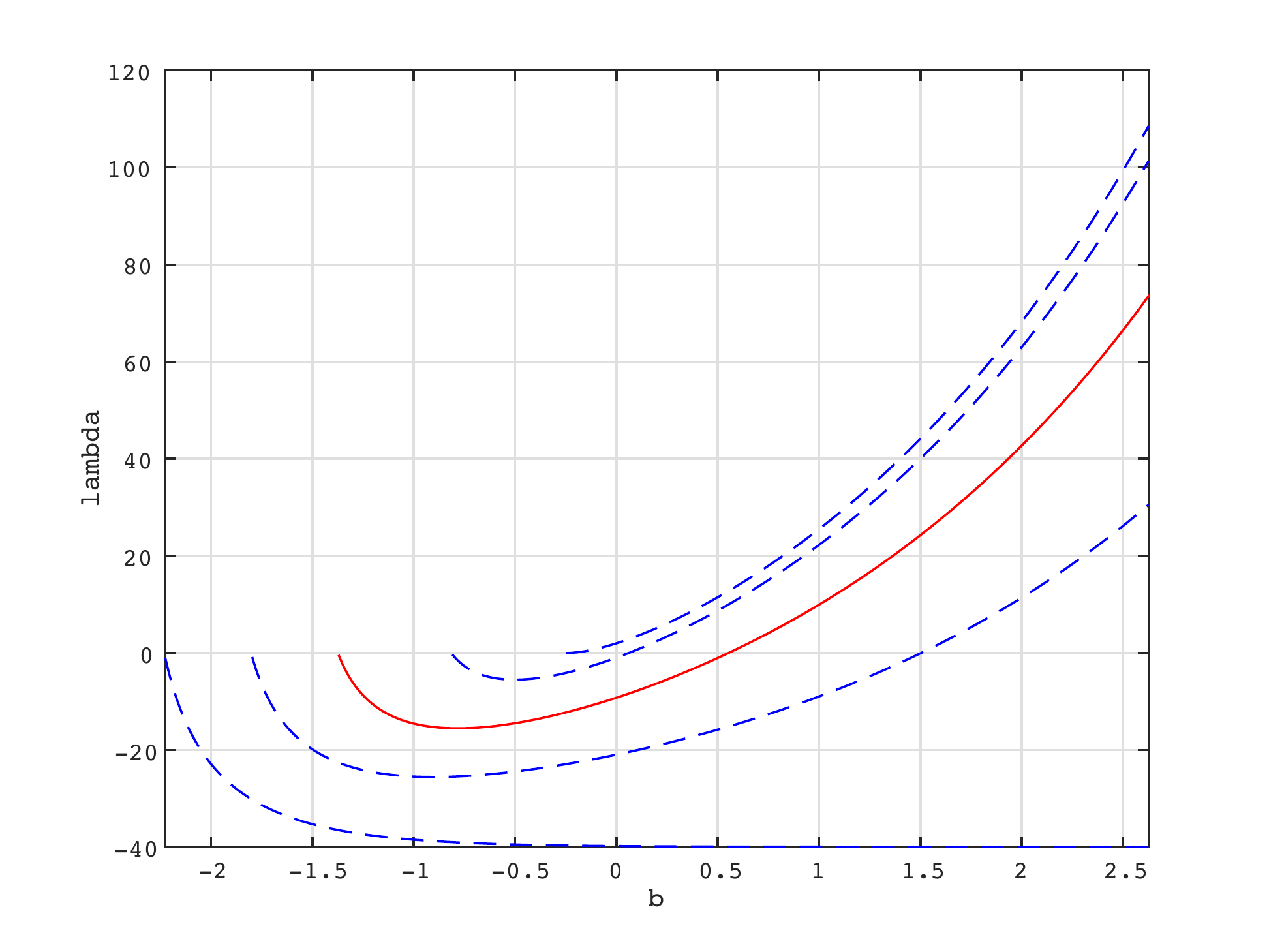}  \\
 $b \mapsto \Lambda(a,b)$ & $b \mapsto \lambda(a,b)$
\end{tabular}
\end{minipage}
\caption{Existence of $(a^*, b^*)$ for Example \ref{example_two_sided_control}. Plots of $b \mapsto \Lambda(a, b)$ on $[a, \infty)$ for the starting values $a = \underline{a}, (\underline{a} + a^*)/2, a^*, (a^* + \overline{a})/2, \overline{a}$ are shown.  The solid curve in red corresponds to the one for $a = a^*$; the point at which $\Lambda(a^*, \cdot)$ is tangent to the x-axis (or $\lambda(a^*, \cdot)$ vanishes) becomes $b^*$. The function $\Lambda(\underline{a},\cdot)$ is monotonically decreasing while $\Lambda(\bar{a},\cdot)$ is monotonically increasing. Equivalently, $\lambda(\underline{a},\cdot)$ is uniformly negative while $\lambda(\bar{a},\cdot)$ is uniformly positive. }  \label{Gamma_plot_quadratic}
\end{center}
\end{figure}

Figure \ref{Gamma_plot_quadratic} shows some sample plots of $b \mapsto \Lambda(a,b)$ and $b \mapsto \lambda(a,b)$. As observed in these plots, we shall show that $a^*$ must lie on $[\underline{a}, \bar{a})$.  

To see this,  when $a \geq \bar{a}$, then $\Lambda(a, \cdot)$ is uniformly positive because $\lambda(a,b) \geq 0$ for $b > a$ in view of \eqref{small_gamma}. 
In addition, by the convergence \eqref{limit_Gamma_a_b} and how $\underline{a}$ is chosen, $\lim_{b \rightarrow \infty }\Lambda(a,b) = \infty$ if $a > \underline{a}$, $\lim_{b \rightarrow \infty }\Lambda(a,b) = -\infty$ if $a < \underline{a}$, and \eqref{limit_Gamma_a_b} vanishes if $a = \underline{a}$.
On the other hand, for any $a < \overline{a}$ and $a < b$,
\begin{align}
\frac \partial {\partial a}\Lambda(a+,b) 
&= -  \tilde{f}'(a+) W^{(q)}(b-a) >  0. \label{derivative_Gamma_a}
\end{align}
This implies that the infimum $a \mapsto \inf_{b > a} \Lambda(a,b)$ is monotonically increasing. Hence, the desired $a^*$ such that $\Lambda(a^*, \cdot)$ touches the x-axis, if it exists, must lie on $(\underline{a}, \bar{a})$. 

By these observations, one can attempt to decrease the value of $a$ starting at $\bar{a}$ until we arrive at (1) a point $a^*$ such that $\inf_{b > a^*} \Lambda(a^*,b) = 0$ or (2) the point $\underline{a}$, whichever comes first.
%
%
For each case, we set $(a^*, b^*)$ as follows.
\begin{enumerate}
\item We set $(a^*, b^*)$ be such that $0 = \inf_{b > a^*} \Lambda(a^*,b) = \Lambda(a^*,b^*)$. Hence, $\mathfrak{C}_a$ holds. If in addition,  $b \mapsto \lambda(a^*, b)$ is continuous at $b^*$, then $\mathfrak{C}_b'$ also holds as well.
\item We set $a^* = \underline{a}$ and $b^* = \infty$.  By \eqref{limit_Gamma_a_b}, $\lim_{b \rightarrow \infty}\Lambda(a^*,b)/W^{(q)}(b-a^*) = 0$, or equivalently $\mathfrak{C}_a$ holds.
\end{enumerate}

\begin{remark} \label{remark_Lambda_positive_a_b}
In Examples \ref{example_capital_injection_dual} and \ref{example_two_sided_control}, by construction, $\Lambda(a^*,x) \geq 0$ for $x \in [a^*, b^*]$.
\end{remark}

\subsection{Variational inequalities and verification} \label{verification_singular}
Below, we shall focus on the case $a^* \in int(\mathcal{I})$ and hence $\mathfrak{C}_a$ is satisfied (this excludes Example \ref{example_capital_injection}): the
value function becomes, by \eqref{expression_v_a_b}, for all $x \leq b^*$,
\begin{align} \label{expression_value_function}
\begin{split}
v_{a^*,b^*} (x) 
&= -C_U R^{(q)}(x-a^*) + \frac {f(a^*)} q Z^{(q)} (x-a^*) - \varphi_{a^*}(x; f) \\
&= -C_U \Big( \frac {\psi'(0+)} q  + x \Big)+ \frac {\tilde{f}(a^*)} q Z^{(q)} (x-a^*) -  \varphi_{a^*} (x;\tilde{f}).
\end{split}
\end{align}
By \eqref{about_gamma} and \eqref{v_derivative_general},
\begin{align}
v_{a^*,b^*}'(x) = -\Lambda(a^*,x) + C_D, \quad a^* \leq x \leq b^*. \label{derivative_value_function}
\end{align}

The verification of optimality requires that our candidate value function $v_{a^*, b^*}$ satisfies the variational inequalities: 
\begin{align} \label{VI_system_double_reflected}
\begin{split}
&(\mathcal{L}-q) v_{a^*,b^*}(x) + f(x)\geq 0,  \quad x \in \mathcal{I}^o, \\
& \min ( v'_{a^*, b^*}(x) + C_U, C_D - v'_{a^*, b^*}(x)) \geq 0, \quad x \in (-\infty, \overline{\mathcal{I}}], \\
&[(\mathcal{L}-q) v_{a^*,b^*}(x) + f(x)]  \min ( v'_{a^*, b^*}(x) + C_U, C_D - v'_{a^*, b^*}(x))= 0, \quad x \in \mathcal{I}^o.
\end{split}
\end{align}
Notice that, when $\underline{\mathcal{I}} > -\infty$, the middle condition is required to hold for the extended set $(-\infty, \overline{\mathcal{I}}]$  because $X$ can jump instantaneously to the region $(-\infty, \underline{\mathcal{I}})$ (and then immediately pushed up to $\mathcal{I}$).  Here, the generator $\mathcal{L} v_{a^*, b^*}$ makes sense due to the smoothness obtained above of $v_{a^*, b^*}$ and because $v_{a^*, b^*}$ is linear below $a^*$ and Assumption \ref{finiteness_X_1} is given.

In order to show that these are sufficient conditions for optimality, in general we need additional assumptions on the tail property of $f$ and the \lev measure.   This is necessary because verification arguments  first localize in order to use It\^o's formula.  After the localization arguments, one needs to  interchange the limits over expectations.  To this end, it is typically required that $|f|$ only increases moderately and/or the \lev measure does not have a heavy tail.

Showing \eqref{VI_system_double_reflected} is the main challenge and the proof needs to be customized for each problem.  However, some inequalities of \eqref{VI_system_double_reflected} are easily shown without strong assumptions on the function $f$.

\begin{lemma} \label{lemma_easy_results_doubly_reflected}Suppose $\mathfrak{C}_a$ holds.
\begin{enumerate}
\item We have $(\mathcal{L}-q)v_{a^*, b^*}(x) + f(x)= 0$ for $a^* < x < b^*$.  
\item  If Assumption \ref{assump_f_tilde} holds with $a^* \leq \bar{a}$, then $(\mathcal{L}-q)v_{a^*,b^*}(x) + f(x) \geq 0$ on $(-\infty, a^*)$.
\item  If  $\Lambda(a^*,x) \geq 0$ for $x \in [a^*, b^*]$, then $v_{a^*,b^*}'(x) \leq C_D$ on $(-\infty, \overline{\mathcal{I}}]$. 
\end{enumerate}
\end{lemma}
\begin{proof}
(1) This is immediate by the results summarized in Section \ref{subsection_martigale_properties} in view of the first equality of \eqref{expression_value_function}.

%

(2) By the second equality of \eqref{expression_value_function}, $v_{a^*, b^*}(x) =  [-C_U {\psi'(0+)}  + {\tilde{f}(a^*)}] / q- C_U x$, for $x < a^*$, and hence $(\mathcal{L}-q)v_{a^*, b^*}(x) + f(x) = \tilde{f}(x) - \tilde{f}(a^*)$.
This is positive by $x \leq a^* < \overline{a}$ and by how $\overline{a}$ is chosen.

(3) In view of \eqref{derivative_value_function}, this inequality holds for $x \in [a^*, b^*]$.  For $x \in (-\infty, a^*)$, we have  $v_{a^*,b^*}'(x) = - C_U$, which is smaller than $C_D$ by \eqref{assump_C_sum}.  Finally, for $x \in (b^*, \infty) \cap \mathcal{I}$, we have  $v_{a^*,b^*}'(x) = C_D$.
\end{proof}

 \begin{figure}[htbp]
\begin{center}
\begin{minipage}{1.0\textwidth}
\centering
\begin{tabular}{c}
 \includegraphics[scale=0.4]{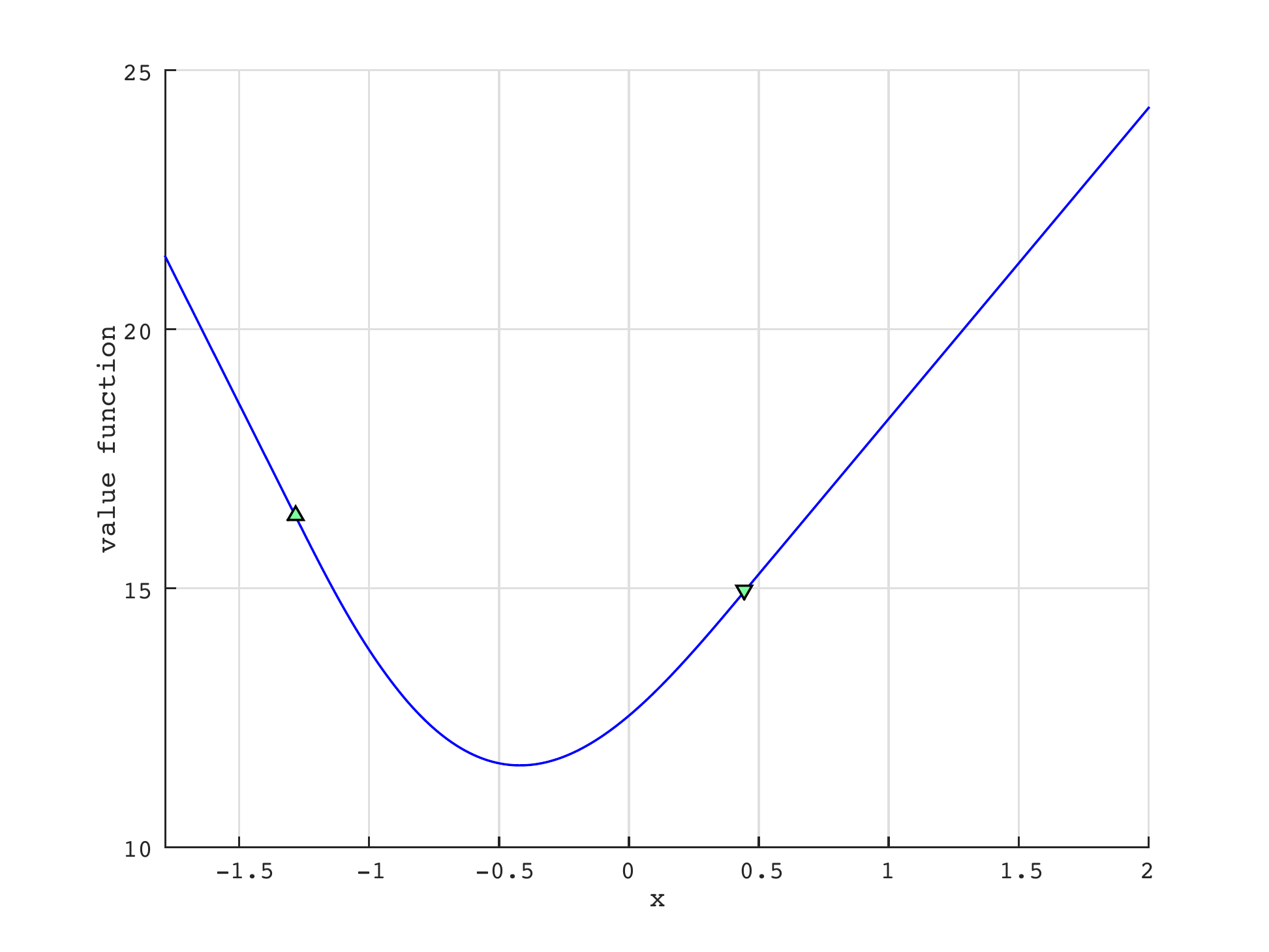}
 \end{tabular}
\end{minipage}
\caption{A sample plot of the value function for Example \ref{example_two_sided_control} when $X$ is of unbounded variation. The up-pointing and down-pointing triangles show the points at $a^*$ and $b^*$, respectively. It can be confirmed that it is twice differentiable at $a^*$ and $b^*$.}  
\end{center}
\end{figure}

For Examples \ref{example_capital_injection_dual} and \ref{example_two_sided_control}, by the fact that $a^* < \bar{a}$ as discussed in Sections \ref{subsection_example_capital_injection} and \ref{subsection_example_capital_injection_dual}, and also by Remark \ref{remark_Lambda_positive_a_b}, the conditions in Lemma \ref{lemma_easy_results_doubly_reflected} hold. Hence, the only pieces left to show in \eqref{VI_system_double_reflected} are 
\begin{enumerate}
\item[(1')] $-C_U \leq v_{a^*,b^*}'(x)$ for all $x \in (a^*,b^*)$,
\item[(2')] $(\mathcal{L}-q)v_{a^*, b^*}(x) + f(x) \geq 0$ for $x  \in (b^*, \infty) \cap \mathcal{I}^o$.  
\end{enumerate}
These conditions unfortunately do not hold generally and must be checked individually.  Here we give brief illustrations on how these hold for Examples \ref{example_capital_injection_dual} and \ref{example_two_sided_control}.

In Example \ref{example_capital_injection_dual}, (1') holds immediately because, with $C_U = -1 < 0$, 
\begin{align*}
v_{a^*,b^*}'(x) = -C_U Z^{(q)}(x-a^*) \geq -C_U. 
\end{align*}
In addition, (2') holds trivially because $(b^*, \infty) \cap \mathcal{I}^o = \varnothing$. 

In Example \ref{example_two_sided_control}, 
thanks to the assumption that $f$ is convex, $x \mapsto \Lambda(a^*, x)$ is first decreasing and decreasing (see Figure \ref{Gamma_plot_quadratic}).  This together with \eqref{derivative_value_function} and the smoothness at $a^*$ and $b^*$, the function $v_{a^*, b^*}$ is convex on $\R$ and hence (1') holds.

The hardest part is to show (2'); the difficulty comes from the fact that the process can jump from $(b^*, \infty)$ to the regions $(-\infty, a^*)$ and $(a^*, b^*)$ where the form of $v_{a^*, b^*}$ changes.  In \cite{Baurdoux_Yamazaki_2015} under the convexity assumption, they use contradiction arguments  similar to \cite{Hernandez_Yamazaki_2013, Loeffen_2008}, where they show, for $x > b^*$,
\begin{align}
(\mathcal{L}-q) (v_{a^*,b^*}-v_{a(x),x})(x-) := \lim_{y \uparrow x}(\mathcal{L}-q) (v_{a^*,b^*}-v_{a(x),x})(y) \geq 0, \label{generator_positive_hypothesis}
\end{align}
where $a(x)$ is the unique value of $a$ such that $\Lambda(a, x) = 0$.
This implies (2') because if both \eqref{generator_positive_hypothesis} and $(\mathcal{L}-q) v_{a^*,b^*}(x) + f(x)< 0$ hold simultaneously, then
\begin{align*}
0 > (\mathcal{L}-q) v_{a^*,b^*}(x) + f(x) \geq (\mathcal{L}-q) v_{a(x),x}(x-) + f(x),
\end{align*}
which contradicts with $ (\mathcal{L}-q) v_{a(x),x}(x-) + f(x) = 0$ that can be shown similarly to Lemma \ref{lemma_easy_results_doubly_reflected}(1).
The proof depends heavily on the convexity of $f$, with which the function $y \mapsto \Lambda(x,y)$ is first decreasing and then increasing. We refer the reader to \cite{Baurdoux_Yamazaki_2015} for more careful analysis.

We conclude this section with a summary of the functions and parameters that played key roles in Examples  \ref{example_capital_injection_dual} and \ref{example_two_sided_control}.   Some similarities and differences with the problems to be considered in later sections can be seen by comparing with Tables \ref{table_impulse} and \ref{table_game} below. 
\begin{table}[ht]
\centering
\begin{tabular}{l l}
\hline
$\Lambda(a,b)$ &$:= C_D + C_U Z^{(q)}(b-a)$ \\
$\tilde{f}'(b)$ & $:=C_U q$ \\ [0.5ex] 
\hline
$a^*$& $:=a^*$ of $(a^*, 0)$ such that  $\mathfrak{C}_a^0$ holds \\
$< \bar{a}$& $:=0=\bar{a}(\tilde{f}')$  \\
$= b^*$ & $:=0 = \overline{\mathcal{I}}$  \\
\hline
\end{tabular}  \\
Example \ref{example_capital_injection_dual}  
\\ \vspace{0.3cm}
%
\begin{tabular}{l l}
\hline
$\Lambda(a,b)$ &$:= C_D + C_U + \varphi_a (b;\tilde{f}')$  \\
$\tilde{f}'(b)$ & $:=f'(b) +C_U q$ \\ [0.5ex] 
\hline
$\underline{a}$& $:=\underline{a}(\tilde{f}')$ \\
$\leq a^*$& $:=a^*$ of $(a^*, b^*)$ such that  $\mathfrak{C}_a$ and $\mathfrak{C}_b$ hold simultaneously\\
$< \bar{a}$& $:=\bar{a}(\tilde{f}')$  \\
$< b^*$ & $:=b^*$ of $(a^*, b^*)$ such that  $\mathfrak{C}_a$ and $\mathfrak{C}_b$ hold simultaneously  \\
\hline
\end{tabular} \\
Example \ref{example_two_sided_control} 
\vspace{0.3cm}
\caption{Summary of the key functions and parameters in Examples  \ref{example_capital_injection_dual} and \ref{example_two_sided_control}. For Example \ref{example_two_sided_control},  when $b^* = \infty$, $a^* =\underline{a}$.}
\label{table_two_sided}
\end{table}

\section{Impulse Control} \label{section_impulse_control}

In impulse control, a strategy $\pi := \left\{ U_t^{\pi}; t \geq 0 \right\}$ is given by $U_t^\pi = \sum_{i: T_i^{\pi} \leq t} u_i^\pi$, $t \geq0$, where $\{ T_i^\pi; i \geq 1 \}$ is an increasing sequence of $\mathbb{F}$-stopping times and $u_i^\pi$,  for $i \geq 1$, is an $\F_{T_i^\pi}$-measurable random variable such that $u_i^\pi \in \mathcal{A}$, $i \geq 1$, a.s.\ for some $\mathcal{A} \subset \R$.

The corresponding controlled process is given by $Y^\pi = \{Y_t^\pi; t \geq 0 \}$ where $Y_{0-}^\pi=0$ and
\begin{align*}
Y_t^\pi := X_t + U_t^\pi, \quad t \geq 0.
\end{align*}
The time horizon is given by $T_{\mathcal{I}^c}^\pi := \inf \{ t > 0: Y_t^\pi \notin \mathcal{I} \}$ for some given closed interval $\mathcal{I}$
and $U^\pi$ must be such that
\begin{align}
Y_{t-}^\pi + \Delta U_t^\pi \in \mathcal{I}, \quad 0 \leq t \leq T_{\mathcal{I}^c}^\pi \quad a.s. \label{admissibility_dividend}
\end{align}
Let $\Pi$ be the set of all admissible strategies.

With $f$, some continuous and piecewise continuously differentiable function  on $\mathcal{I}$, and $q > 0$,
the problem is to compute the value function
\begin{align*}
v(x) := \inf_{\pi \in \Pi}v^\pi (x) 
\end{align*}
where
\begin{align*}
v^\pi (x) 
&:= \E_x \Big[ \int_0^{T_{\mathcal{I}^c}^\pi} e^{-qt} f (Y^{\pi}_{t}) \diff t +  \sum_{0 \leq t \leq T_{\mathcal{I}^c}^\pi} e^{-q t} [C_U | \Delta U_t^{\pi} | +K] 1_{\{ |\Delta U_t^{\pi}| > 0 \}}\Big], \quad x \in \R, 
\end{align*}
and to obtain an admissible strategy that minimizes it, if such a strategy exists.  
The constant $C_U$ is the \emph{proportional cost}, which is not necessarily restricted to be a positive value.  On the other hand, 
$K$ is the \emph{fixed cost} and must be strictly positive.  Again in this section, we assume Assumption \ref{finiteness_X_1} (note that this is not necessarily needed for Example \ref{example_dividedn_fixed_costs} below).


\begin{example} \label{example_dividedn_fixed_costs}In the optimal dividend problem with fixed costs driven by a spectrally negative \lev process, each time dividend is paid, a fixed cost $K$ is incurred.  In addition, the problem is terminated at ruin (i.e.\ $\mathcal{I} = [0, \infty)$).  The condition \eqref{admissibility_dividend} means that one cannot pay more than the remaining surplus.

The objective is to maximize the total expected discounted dividends minus that for fixed costs.  We can formulate this as a minimization problem as above by setting $C_U = -1$,  $U_t^\pi$ being the negative of the cumulative amount of dividends until $t \geq 0$, and $\mathcal{A} = (-\infty, 0)$.  Here, $f$ is assumed to be zero.
  This problem has been solved by Loeffen \cite{Loeffen_2009_2}  for a spectrally negative \lev process with a completely monotone \lev density.
\end{example}

\begin{example} \label{example_dividedn_fixed_costs_dual} In the dual model of Example \ref{example_dividedn_fixed_costs}, it is assumed that the underlying process is a spectrally positive \lev process.
By flipping the processes  with respect to the origin, it is easy to see that it is equivalent to the above formulation driven by a spectrally negative \lev process with $\mathcal{A} = (0, \infty)$, $\mathcal{I}= (-\infty, 0]$ and $C_U = -1$. This problem has been solved by  Bayraktar et al.\ \cite{Bayraktar_2013} for a general spectrally positive \lev process.
%
%
\end{example}

\begin{example} \label{example_inventory_control}Continuous-time inventory control often uses this model.  Here, the function $f$ corresponds to the cost of holding and shortage when $x > 0$ and $x < 0$, respectively.  With the assumption that backorders are allowed, the problem is infinite-horizon ($\mathcal{I} = \R$).
Bensoussan et al.\ \cite{Bensoussan_2009, Bensoussan_2005} considered the case of a spectrally negative compound Poisson process perturbed by a Brownian motion with $\mathcal{A} = (0, \infty)$.  It has been generalized by Yamazaki \cite{Yamazaki_2013} to a general spectrally negative \lev model.  As in Example \ref{example_two_sided_control}, we assume that $f$ is convex.
\end{example}

\subsection{The $(s,S)$-strategy} With the fixed cost $K > 0$ incurred each time the control $U^\pi$ is activated, it is clear that the reflection strategy is no longer feasible; instead one needs to solve the tradeoff between controlling the process and minimizing the number of activation of $U^\pi$.  In this sense, the $(s,S)$-strategy is a natural candidate for an optimal strategy: whenever the process goes below (resp.\ above) a level $s$, it pushes the process up (resp.\ down) to $S$ when $s < S$ (resp.\ $S < s$).

Suppose $\pi^{s,S} := \{ U^{s,S}_t; t \geq 0 \}$ is the $(s,S)$-strategy, and $Y^{s,S}$ and $T^{s,S}_{\mathcal{I}^c}$ are the corresponding controlled process and the termination time, respectively.  By using the results summarized in Section \ref{fluctuations_underlying}, it is a simple exercise to compute the corresponding expected NPV of costs:
\begin{align} \label{def_v_c}
v_{s,S} (x) 
&:= \E_x \Big[ \int_0^{T^{s,S}_{\mathcal{I}^c}} e^{-qt} f (Y^{s,S}_{t}) \diff t +  \sum_{0 \leq t \leq T^{s,S}_{\mathcal{I}^c}} e^{-q t} [C_U | \Delta U_t^{s,S}| +K] 1_{\{ |\Delta U_t^{s,S}| > 0 \}}\Big], \quad x \in \R.
\end{align}
To see this, for the case $s < S$, it is noted (from the construction of the process $Y^{s,S}$) that  $\p_x$-a.s.,  $Y_t^{s,S} = X_t$ for $0 \leq t < T_s^-$ 
 and $\Delta U_{T^-_s}^{s,S}   = S- X_{T^-_s}$ on $\{ T^-_s < T_{\mathcal{I}^c}^{s,S} \}$.
By these and  the strong Markov property of  $Y^{s,S}$,
the expectation \eqref{def_v_c} must satisfy, for every  $x > s$,
\begin{align}
v_{s,S} (x) &= \E_x \Big[ \int_0^{T_{s}^- \wedge T^{s,S}_{\mathcal{I}^c}} e^{-qt}f(X_{t}) \diff t \Big] + \E_x \left[ e^{-q T_{s}^-}  (C_U (S - X_{T_{s}^-})+K) 1_{\{T_s^- < T^{s,S}_{\mathcal{I}^c}\}}  \right] \\ &+ \E_x \left[ e^{-q T_{s}^-} 1_{\{T_s^- < T^{s,S}_{\mathcal{I}^c} \}}   \right] v_{s,S}(S).   \label{v_recursion}
\end{align}
Here the expectations on the right hand side can be computed by the identities given in Section \ref{section_review_levy}.
By setting $x = S$ on both sides, we can solve for $v_{s,S}(S)$; substituting this back in, we obtain $v_{s,S}(x)$ for $x \in \R$.  In particular, for the computation when $\mathcal{I} = \R$, see \eqref{QVI_system} below.

The case $s > S$ is even simpler because then there is no overshoot at the time it reaches $s$:  we have, for $x < s$,
\begin{align*}
v_{s,S} (x) &= \E_x \Big[ \int_0^{T_{s}^+ \wedge T^{s,S}_{\mathcal{I}^c}} e^{-qt}f(X_{t}) \diff t \Big] 
+ \E_x \left[ e^{-q T_{s}^+} 1_{\{T_s^+ < T^{s,S}_{\mathcal{I}^c} \}}  \right] [v_{s,S}(S) + C_U (s-S)+K].   
\end{align*}
We can similarly obtain first $v_{s,S}(S)$ and then, by substituting this back in, $v_{s,S}(x)$, for $x \in \R$. See, e.g., \cite{Loeffen_2009_2} for explicit expressions when $f \equiv 0$.

\begin{remark} The same technique can be used to compute also the two-sided extension (i.e.\ $\mathcal{A} = \R \backslash \{0\}$) of the $(s,S)$-strategy: in this case, the strategy is specified by four parameters, say, $(d,D,U,u)$. The controller pushes the process up to $D$ as soon as it goes below $d$
and pushes down to $U$ as soon as it goes above $u$, while he does not intervene whenever it is within the set $(d, u)$.  See \cite{Yamazaki_cash} for the fluctuation identities.
\end{remark}

\subsection{Smoothness of the value function}  \label{smoothness_impulse}
Focusing on the set of $(s,S)$-strategies, the first step again is to narrow down to a candidate optimal strategy by deciding on the values of $s$ and $S$, which we call $s^*$ and $S^*$.  Again, as there are two values to be identified, naturally we need two equations to identify these.

\begin{enumerate}
\item As is clear from what we have seen in the previous section, the value function is expected to satisfy some continuity/smoothness at the point $s^*$.  In comparison to the case of singular control, \emph{the degree of smoothness is decreased by one} in the case of impulse control.  This can be summarized as follows:

When $s^* < S^*$ (where $v_{s^*,S^*}$  is linear below $s^*$ and hence $v_{s^*,S^*}'(s^*-) = -C_U$),
\begin{enumerate}
\item if $s^*$ is regular for $(-\infty, s^*)$ (or equivalently $X$ is of unbounded variation), then the continuous differentiability at $s^*$ is expected;
\item if $s^*$ is irregular for $(-\infty, s^*)$ (or equivalently $X$ is of bounded variation), then the continuity at $s^*$ is expected.
\end{enumerate}

When $s^* > S^*$ (where $v_{s^*,S^*}$  is linear above $s^*$ and hence $v_{s^*,S^*}'(s^*+) = C_U$), because $s^*$ is regular for $(s^*, \infty)$ for any spectrally negative \lev process, the  continuous differentiability at $s^*$ is expected. 

It is noted that alternatively one can use the first-order condition on $s^*$ so that $\partial v_{s,S}/ \partial s |_{s = s^*, S= S^*}$ vanishes: we typically arrive at the same equation.
\item The other equation can be obtained by what we postulate at the point $S^*$. This is less intuitive than (1). However, if we consider the first-order condition at $S^*$ so that $\partial v_{s,S}/ \partial S |_{s=s^*, S=S^*}$ vanishes, easy computation derives that it tends to be equivalent to the condition $v_{s^*,S^*}'(S^*) = -C_U$ (resp.\ $v_{s^*,S^*}'(S^*) = C_U$) when $s^* < S^*$ (resp.\ $s^* > S^*$). 
\end{enumerate}

From the above discussions, when $s^* < S^*$, except for the case $X$ is of bounded variation, we arrive at the function that satisfies
\begin{align*}
v_{s^*,S^*}'(s^*) = v_{s^*,S^*}'(S^*) = -C_U.
\end{align*}
Due to this fact, it is often easier if we deal  with a modified function 
\begin{align}
\tilde{v}_{s,S}(x) := v_{s,S}(x) + C_U x; \label{def_v_tilde}
\end{align}
by this, some terms tend to disappear and computation gets simplified.
  When $S^* < s^*$, then the sign of the coefficient of $C_U$ is flipped.

In impulse control, while the two equations that identify the two unknown parameters $(s^*,S^*)$ are slightly different from the singular control case for $(a^*,b^*)$ as in Section  \ref{smoothness_double_barrier}, we shall see that these two equations possess a similar relation to those obtained for $(a^*,b^*)$.  Namely, the desired pair $(s^*,S^*)$  is such that a function of two variables and its partial derivative with respect to one of the parameters vanish simultaneously.

\subsubsection{The case of Example \ref{example_inventory_control}} \label{inventory_existence_example}
For Example \ref{example_inventory_control}, we shall see that the desired $(s^*, S^*)$ are those $(s,S)$ such that 
\begin{align}
\mathfrak{C}_s&: \frac {\Lambda(s,S)} {\overline{\Theta}^{(q)}(S-s)} = 0,  \label{G_zero}\\
\mathfrak{C}_S&: \frac {\Theta^{(q)}(S-s)} {\overline{\Theta}^{(q)}(S-s)}\Lambda(s,S)  -\lambda(s,S) = 0, \label{H_zero}
\end{align}
where $\Theta^{(q)}$ is as defined in \eqref{def_theta} with its antiderivative $\overline{\Theta}^{(q)}$ given by 
\begin{align*}
\overline{\Theta}^{(q)}(x) &:= W^{(q)}(x)    - \Phi(q) \overline{W}^{(q)}(x) > 0,
\end{align*}
and 
\begin{align}  
\label{def_G}
\Lambda(s,x) 
&:= \Phi(q)   \Psi(s;\tilde{f}) \overline{W}^{(q)}(x-s) + K -  \varphi_s(x; \tilde{f}), \quad x, s \in \R, \\ \label{def_H}
\lambda(s,x) &:= \frac \partial {\partial x} \Lambda(s,x), \quad x > s.
\end{align} 

Here, we shall confirm briefly how this is so.  Note that when $\mathfrak{C}_s$ is satisfied, then $\mathfrak{C}_S$ is equivalent to the condition:
\begin{align}
\mathfrak{C}_S'&: \lambda(s,S) = 0. \label{H_zero_simple}
\end{align}
\begin{remark}
We note the similarity between $\mathfrak{C}_s$ and $\mathfrak{C}_S$ (or $\mathfrak{C}'_S$) with the conditions $\mathfrak{C}_a$ and $\mathfrak{C}_b$ (or $\mathfrak{C}'_b$) as in \eqref{smoothness_condition1}, \eqref{smoothness_condition2} (or \eqref{smoothness_condition2_prime}) in the two-sided singular control case. \end{remark}

First, by using the technique (using the equation \eqref{v_recursion}) discussed above, we can compute \eqref{def_v_tilde}: for all $s < S$, \begin{align}
\begin{split}
\tilde{v}_{s,S} (S)   &= \frac  {\Phi(q)} {q  \overline{\Theta}^{(q)}(S-s)}  \left[ \overline{\Theta}^{(q)}(S-s) \left[ \Psi(s;\tilde{f}) - \frac q {\Phi(q)} \left( K + \frac {C_U \psi'(0+)} q \right) \right] + \Lambda(s,S) \right], \\
\tilde{v}_{s,S} (x) 
&=  \left\{ \begin{array}{ll} - \frac {\overline{\Theta}^{(q)}(x-s)} {\overline{\Theta}^{(q)}(S-s)}\Lambda(s,S)  + \Lambda(s,x)  + \tilde{v}_{s,S}(S), & x \geq s, \\
K +  \tilde{v}_{s,S}(S), & x < s. \end{array} \right. 
\end{split} \label{v_tilde_with_G}
\end{align}

Differentiating \eqref{v_tilde_with_G},
\begin{align} \label{derivative_with_G}
\tilde{v}_{s,S}' (x) 
&= - \frac {\Theta^{(q)}(x-s)} {\overline{\Theta}^{(q)}(S-s)}\Lambda(s,S)  + \lambda(s,x), \quad  s <  x < S.
\end{align}

From these expressions, we shall see that the conditions $\mathfrak{C}_s$ and $\mathfrak{C}_S$ as in \eqref{G_zero} and \eqref{H_zero} guarantee the desired smoothness/slope conditions described above: namely,
\begin{enumerate}
\item $\tilde{v}_{s^*,S^*}(\cdot)$ is continuous (resp.\ differentiable) at $s^*$ when $X$ is of bounded (resp.\ unbounded) variation, 
\item $\tilde{v}_{s^*,S^*}' (S^*) = 0$.
\end{enumerate}

(1) Regarding the continuity at $s$, by \eqref{v_tilde_with_G},
\begin{align*}
\tilde{v}_{s,S} (s+) 
&= - \frac {\overline{\Theta}^{(q)}(0)} {\overline{\Theta}^{(q)}(S-s)}\Lambda(s,S)  +K + \tilde{v}_{s,S}(S) = - \frac {\overline{\Theta}^{(q)}(0)} {\overline{\Theta}^{(q)}(S-s)}\Lambda(s,S)  + \tilde{v}_{s,S}(s-),
\end{align*}
where $\overline{\Theta}^{(q)}(0) = 0$ if and only if $X$ is of unbounded variation in view of  \eqref{eq:Wq0}.  Hence, the continuity at $x = s$ holds if and only if $\mathfrak{C}_s$ holds for the case of bounded variation. On the other hand, it holds automatically for the unbounded variation case. 

For the case of unbounded variation, we further pursue the differentiability at $x=s$.  The equation \eqref{derivative_with_G} gives $\tilde{v}_{s,S}' (s+) = - \frac {\Theta^{(q)}(0)} {\overline{\Theta}^{(q)}(S-s)}\Lambda(s,S)$, and hence $\mathfrak{C}_s$ leads to the differentiability at $s$.


(2) Regarding the slope condition at $S$, we have $\tilde{v}_{s,S}' (S) 
=   - \frac {\Theta^{(q)}(S-s)} {\overline{\Theta}^{(q)}(S-s)}\Lambda(s,S)  + \lambda(s,S)$.
Hence $\mathfrak{C}_S$ guarantees $\tilde{v}_{s,S}' (S) = 0$ as desired.




\underline{Existence of $(s^*, S^*)$}: 
We now illustrate how the existence of $(s^*, S^*)$ that satisfy $\mathfrak{C}_s$ and $\mathfrak{C}_S$ can be shown.  
Here, as in Example \ref{example_two_sided_control}, we shall assume Assumption \ref{assump_f_tilde}: then,
\begin{align*}
\underline{a} \equiv \underline{a}(\tilde{f}') \quad \textrm{and} \quad \bar{a} \equiv \bar{a}(\tilde{f}')
\end{align*}
are well-defined and finite as in the discussion given in Section \ref{two_sided_control_existence_example}.

We shall see that the desired $s^*$ lies on the left of $\underline{a}$ while $S^*$ lies on its right.  As $K$ decreases, the distance between $s^*$ and $S^*$ is expected to shrink and converge to $\underline{a}$, which is the optimal barrier in Example \ref{example_two_sided_control} for the case $b^* =\infty$.

To show the existence of $(s^*, S^*)$, we shall first write
\begin{align} \label{G_H_different_expressions}
\begin{split}
\Lambda(s,S)  &=  \int_{s}^{S}  \Psi(y; \tilde{f}') \overline{\Theta}^{(q)}(S-y)  \diff y  + K,  \quad s, S \in \R, \\
\lambda(s,S)
&= \Psi(S; \tilde{f}')  W^{(q)}(0)  + \int_{s}^S \Psi(y; \tilde{f}')  \Theta^{(q)}(S-y)   \diff y, \quad S > s.
\end{split}
\end{align}
In Figure \ref{G_plot}, we show sample plots of the functions $S \mapsto \Lambda(s,S)$ and $S \mapsto \lambda(s,S)$ for several values of starting points $s$, including $\underline{a}$ and $a^*$.  

As can be confirmed by the figure and also clear from \eqref{G_H_different_expressions}, by how $\underline{a}$ is chosen, we have the following properties:
\begin{enumerate}
\item When $s > \underline{a}$, $\lambda(s,S) > 0$ for $S > s$ and hence $S \mapsto \Lambda(s,S)$ is monotonically increasing on $[s, \infty)$.
\item When $s < \underline{a}$, $\partial \Lambda(s,S) / {\partial s}  =  -\Psi(s; \tilde{f}') \overline{\Theta}^{(q)}(S-s) \geq 0$ by how $\underline{a}$ is chosen.
\item For every fixed $s \in \R$,  $\lim_{S \uparrow \infty}\Lambda(s,S) = \infty$.  \item For every fixed $S \in \R$,  $\lim_{s \downarrow -\infty}\Lambda(s,S) = -\infty$.
\item For any $s \in \R$, $\Lambda(s,s) = K > 0$. 
\end{enumerate}

\begin{figure}[htbp]
\begin{center}
\begin{minipage}{1.0\textwidth}
\centering
\begin{tabular}{cc}
 \includegraphics[scale=0.4]{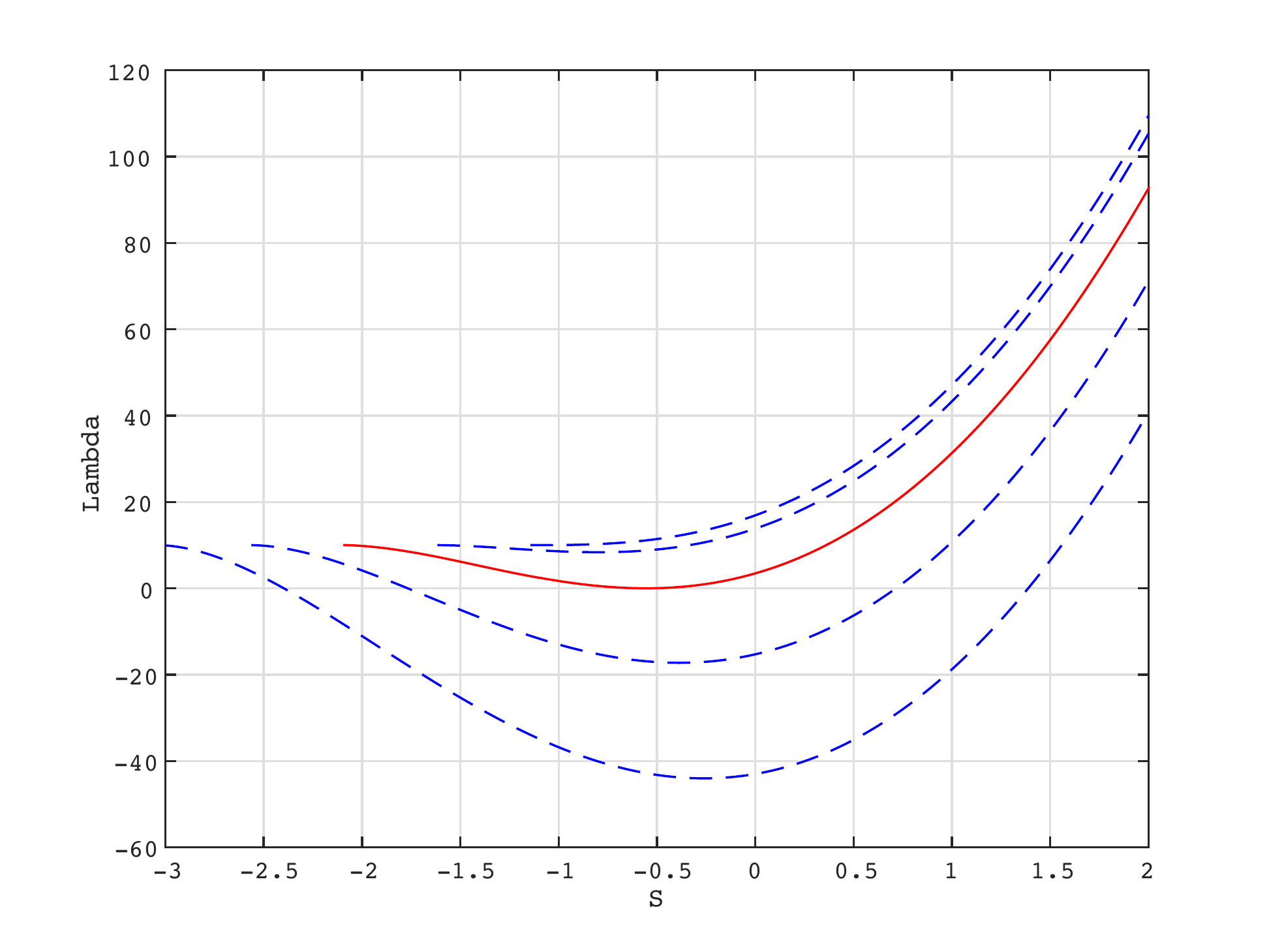} & \includegraphics[scale=0.4]{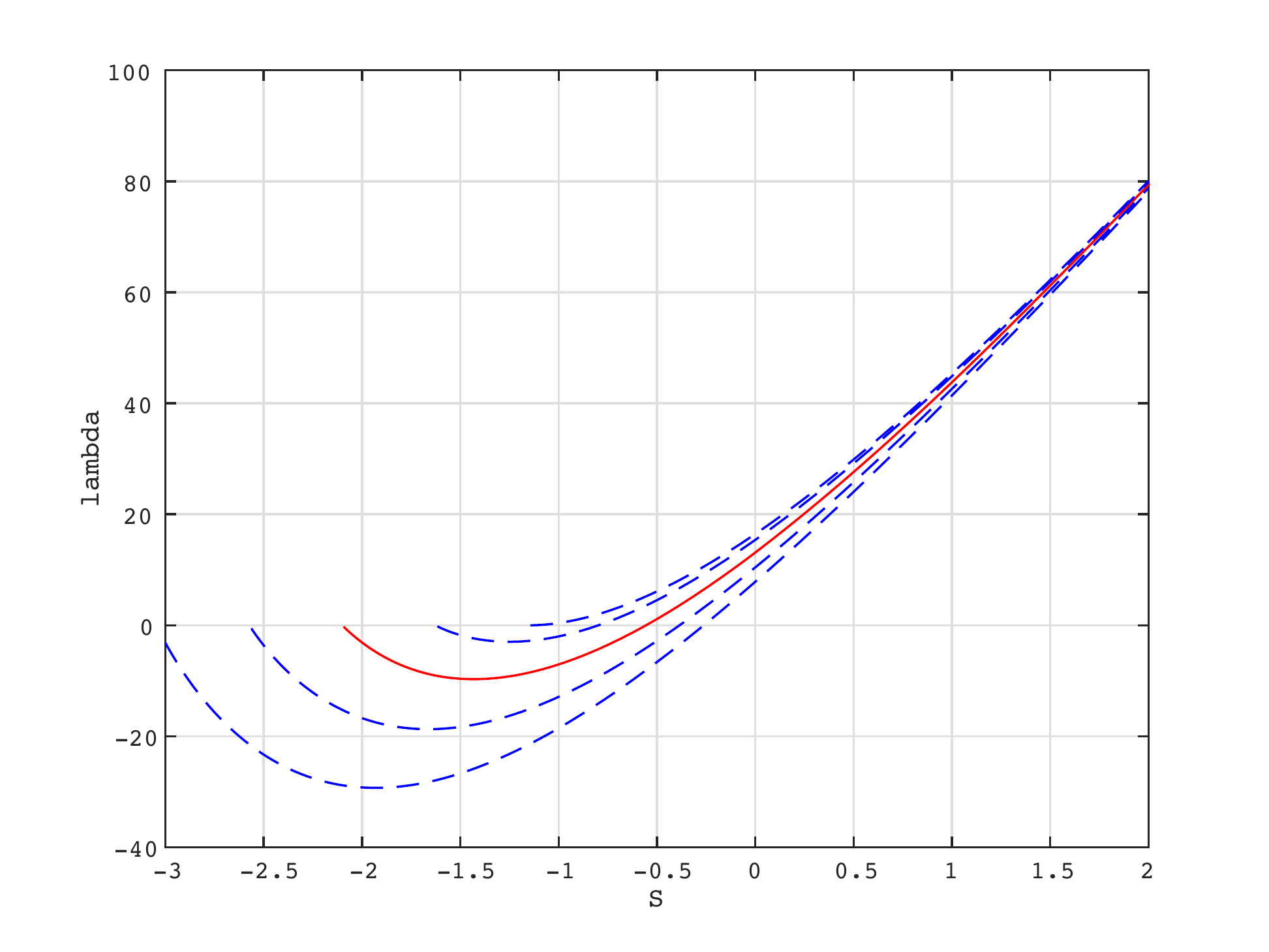}  \\
 $S \mapsto \Lambda(s,S)$ & $S \mapsto \lambda(s,S)$
\end{tabular}
\end{minipage}
\caption{Existence of $(s^*, S^*)$ for Example \ref{example_inventory_control}. Plots of $S \mapsto \Lambda(s, S)$ and $S \mapsto \lambda(s, S)$ on $[s, \infty)$ for five values of $s$ are shown.    The line in red corresponds to the one for $s = s^*$; the point at which $\Lambda(s^*, \cdot)$ is tangent to the x-axis becomes $S^*$. The rightmost curve corresponds to the one with $s = \underline{a}$; it is confirmed that $\Lambda(\underline{a}, \cdot)$ is monotonically increasing and $\lambda(\underline{a}, \cdot)$ is uniformly positive.}  \label{G_plot}
\end{center}
\end{figure}

It is now clear how to obtain the desired $(s^*, S^*)$.  Similarly to Example \ref{example_two_sided_control}, starting at $s = \underline{a}$, we decrease the value of $s$ until we arrive at $s^*$ such that $\inf_{S > s^*} \Lambda(s^*,S)=0$. This exists because the function $s \mapsto \inf_{S > s}\Lambda(s,S)$, $s < \underline{a}$, is increasing by the property (2) above and goes to $-\infty$ as $s \downarrow -\infty$ by the property (4). Note that, because  \eqref{G_H_different_expressions} implies $\lambda(s^*,S) < 0$ for $S \in (s^*, \underline{a})$, we must have $S^* > \underline{a}$.
Because $\inf_{S > s^*} \Lambda(s^*,S)=0$ attains a local minimum at $S=S^*$, we must have $\lambda(s^*,S^*) = \Lambda(s^*,S^*) = 0$, as desired.

\subsubsection{Brief remarks on the cases of Examples \ref{example_dividedn_fixed_costs} and \ref{example_dividedn_fixed_costs_dual}.}  In \cite{Loeffen_2009_2} and \cite{Bayraktar_2013}, they use the first-order conditions to obtain $(s^*, S^*)$ in Examples \ref{example_dividedn_fixed_costs} and \ref{example_dividedn_fixed_costs_dual}, respectively.  To this end, they used the argument that the surface $(s,S) \mapsto v_{s,S}(x)$ has a global minimum (if formulated as a minimization problem).

The difficulty in their case is that because $\mathcal{I}$ has a finite boundary $0$, it can happen that $S^*$ (or both $s^*$ and $S^*$) is zero.  This means that the $(s^*, S^*)$-strategy, once activated, moves the controlled process to the default boundary.  In Example \ref{example_dividedn_fixed_costs_dual} where $0$ is regular for ${\mathcal{I}}^c = (0, \infty)$, ruin then occurs immediately.  On the other hand, in Example \ref{example_dividedn_fixed_costs}, it is regular for $\mathcal{I}^c = (-\infty, 0)$ if and only if $X$ is of unbounded variation. Hence, while ruin occurs immediately for the unbounded variation case, it stays above $0$ for a positive amount of time a.s.  This suggests one difficulty in solving the spectrally negative \lev case.

If $S^* \neq 0$, then for both problems,  the slope condition $v_{s^*, S^*}' (S^*) = - C_U = 1$ (resp.\ $v_{s^*, S^*}' (S^*) =  C_U = -1$) is satisfied for Example \ref{example_dividedn_fixed_costs_dual} (resp.\ Example \ref{example_dividedn_fixed_costs}).  Similarly, if $s^* \neq 0$, then
the smoothness condition  $v_{s^*, S^*}' (s^*) = - C_U = 1$ (resp.\ $v_{s^*, S^*}' (s^*) =  C_U = -1$) is satisfied for Example \ref{example_dividedn_fixed_costs_dual} (resp.\ Example \ref{example_dividedn_fixed_costs}).

\subsection{Quasi-variational inequalities and verification}  

The verification of optimality requires that the candidate value function $v_{s^*, S^*}$ satisfies the QVI (quasi-variational inequalities):
\begin{align} \label{QVI_system}
\begin{split}
&(\mathcal{L}-q) v_{s^*,S^*}(x) + f(x)\geq 0,  \quad x \in \mathcal{I}^o \backslash \{s^*\},\\
&v_{s^*,S^*}(x) \leq K + \inf_{u \in \mathcal{A}, x+ u \in \mathcal{I}} \left[ C_U |u| + v_{s^*,S^*}(x+u) \right],  \quad x \in \mathcal{I}, \\
&[(\mathcal{L}-q) v_{s^*,S^*}(x) + f(x)] \big[v_{s^*,S^*}(x) - K - \inf_{u \in \mathcal{A}, x+ u \in \mathcal{I}} \left[ C_U |u| + v_{s^*,S^*}(x+u) \right] \big] = 0, \quad x \in \mathcal{I}^o \backslash \{s^*\}.
\end{split}
\end{align}
For its proof, 
 see \cite{Bensoussan_Lions_1984, Bensoussan_2005}.  Similarly to the singular control case, in general we need additional assumptions on the tail growth of $f$ and the \lev measure. In particular, in  \cite{Bensoussan_2005, Yamazaki_2013}, it is assumed that the growth of $f$ in the tail is at most polynomial.
 
 \subsubsection{The case of Example \ref{example_inventory_control}}

With $(s^*, S^*)$ that satisfy $\mathfrak{C}_s$, the function \eqref{v_tilde_with_G} simplifies to, for $x \in \R$,
\begin{align} 
\tilde{v}_{s^*,S^*}(S^*) &= \frac {\Phi(q)} q \Psi(s^*; \tilde{f})  - K  - \frac  {C_U \psi'(0+)} q, \label{v_tilde_final_fixed} \\
\tilde{v}_{s^*,S^*} (x) 
&= \Lambda(s^*,x)  + \tilde{v}_{s^*,S^*} (S^*),  \label{v_tilde_final}
\end{align}
or equivalently
\begin{align} \label{v_tilde_optimal}
\begin{split}
v_{s^*,S^*} (x) 
&= \left( \frac {\Phi(q)} q  \Psi(s^*; f) + \frac {C_U} {\Phi(q)}\right) Z^{(q)}(x-s^*)   - C_U R^{(q)}(x-s^*) - \varphi_{s^*} (x; f). 
\end{split}
\end{align}
See Figure \ref{fig_value_function_impulse} for a sample plot of $v_{s^*, S^*}$.

 \begin{figure}[htbp]
\begin{center}
\begin{minipage}{1.0\textwidth}
\centering
\begin{tabular}{c}
 \includegraphics[scale=0.4]{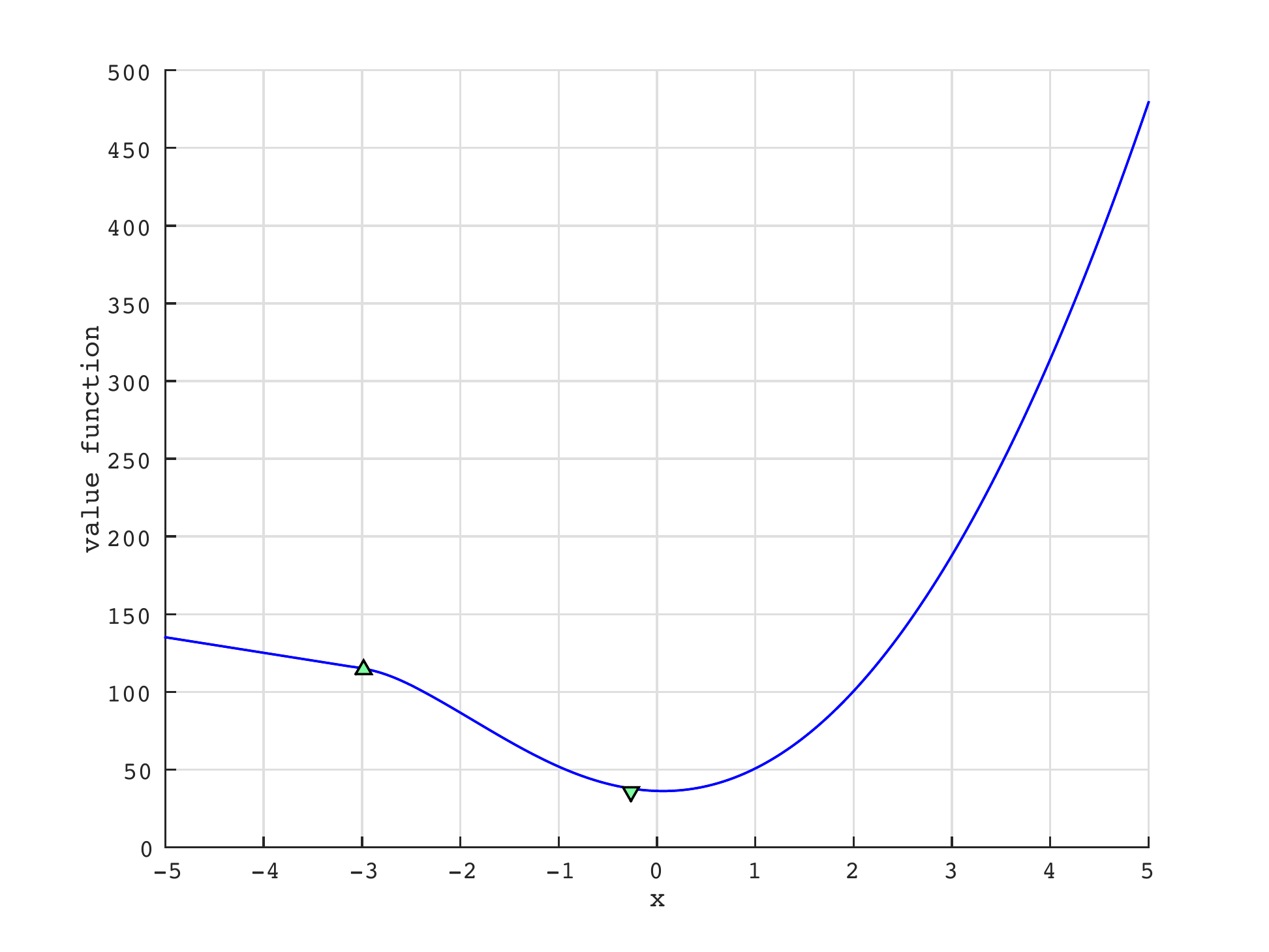}
 \end{tabular}
\end{minipage}
\caption{A sample plot of the value function $v_{s^*, S^*}$ for Example \ref{example_inventory_control} when $X$ is of unbounded variation. The up-pointing and down-pointing triangles show the points at $s^*$ and $S^*$, respectively.}  \label{fig_value_function_impulse}
\end{center}
\end{figure}
 
 Similarly to the singular control case (see Lemma \ref{lemma_easy_results_doubly_reflected}), some inequalities of \eqref{QVI_system} are easily shown with minor  assumptions on the function $f$.

\begin{lemma} \label{lemma_easy_results_impulse}Suppose $\mathfrak{C}_s$ holds.
\begin{enumerate}
\item We have $(\mathcal{L}-q)v_{s^*, S^*}(x) + f(x)= 0$ for $x > s^*$.  
\item  If Assumption \ref{assump_f_tilde} holds and $\underline{a}$ is well-defined and finite with $s^* \leq \underline{a} < \bar{a}$, then $(\mathcal{L}-q)v_{s^*,S^*}(x) + f(x) \geq 0$ on $(-\infty, s^*)$.
\end{enumerate}
\end{lemma}
\begin{proof}
(1) In view of \eqref{v_tilde_optimal}, this is immediate by the results summarized in Section \ref{subsection_martigale_properties}.

%

(2) 
Because $\tilde{v}_{s^*, S^*}(x) = K + \tilde{v}_{s^*, S^*} (S^*)$ for $x < s^*$ and by  \eqref{v_tilde_final_fixed},
\begin{align*}
(\mathcal{L}-q) v_{s^*,S^*}(x) + f(x)  = - q (K + \tilde{v}_{s^*, S^*} (S^*)) - C_U\psi'(0+) + C_Uq x + f(x) = \tilde{f}(x) - \tilde{f}(s^*) - \Psi(s^*; \tilde{f}').
\end{align*}
This is positive by $x < s^* < \underline{a} \leq \bar{a}$ and how $\underline{a}$ and $\bar{a}$ are chosen.
\end{proof}

%


In view of Lemma \ref{lemma_easy_results_impulse}, the remaining task is to show that
\begin{align} \label{verification_bensoussan_difficult}
\begin{split}
v_{s^*,S^*}(x) &= K + \inf_{u \geq 0} \left[ C_Uu + v_{s^*,S^*}(x+u) \right], \quad x \leq s^*, \\
v_{s^*,S^*}(x) &\leq K + \inf_{u \geq 0} \left[ C_U u + v_{s^*,S^*}(x+u) \right], \quad x > s^*.
\end{split}
\end{align}
These can be shown for $x \leq \underline{a}$ easily as follows.
For $x \leq s^*$, in view of \eqref{v_tilde_final} and because $S^*$ minimizes $\Lambda(s^*,x)$ over $x \in \R$, we must have 
\begin{align} \label{v_s_S_minimum}
\tilde{v}_{s^*,S^*}(S^*) = \inf_{x \in \R} \tilde{v}_{s^*,S^*}(x).
\end{align}
Hence, 
\begin{align}
\tilde{v}_{s^*,S^*}(x)  = \tilde{v}_{s^*,S^*}(s^*) = \tilde{v}_{s^*,S^*}(S^*)+K = K + \inf_{u \geq 0} \left[ C_U u + v_{s^*,S^*}(x+u) \right], \quad x \leq s^*. \label{v_s_S_less_than_s}
\end{align}
The case $s^* \leq x \leq \underline{a}$ also holds by \eqref{v_s_S_minimum} and because $\tilde{v}_{s^*,S^*}' (x) = \lambda(s^*, x) < 0$ on $[s^*, \underline{a}]$ in view of how $\underline{a}$ is chosen and \eqref{G_H_different_expressions}.    

Unfortunately, the proof of \eqref{verification_bensoussan_difficult} for $x > \underline{a}$ is difficult and, we need a nonstandard technique.  As the fluctuation theory and scale function do not simplify the proof to our best knowledge, it is out of scope of this note.  We refer the reader to the proof of Theorem 1(iii) of Benkherouf and Bensoussan \cite{Bensoussan_2009}.

Below, we summarize the functions and parameters that played important roles in characterizing the optimal solution in Examples  \ref{example_inventory_control}.

\begin{table}[ht]
\centering
\begin{tabular}{l l}
\hline
$\Lambda(a,b)$ &$:= \Phi(q)   \Psi(s;\tilde{f}) \overline{W}^{(q)}(S-s) + K -  \varphi_s(S; \tilde{f})$  \\ $\tilde{f}'(b)$ & $:=f'(b) +C_U q$ \\ [0.5ex] 
\hline
$s^*$& $:=s^*$ of $(s^*, S^*)$ such that  $\mathfrak{C}_s$ and $\mathfrak{C}_S$ hold simultaneously\\
$< \underline{a}$& $:=\underline{a}(\tilde{f}')$ \\
$< S^*$ & $:=S^*$ of $(s^*, S^*)$ such that  $\mathfrak{C}_s$ and $\mathfrak{C}_S$ hold simultaneously  \\
\hline
\end{tabular} \vspace{0.3cm}
\caption{Summary of the key functions and parameters in Example \ref{example_inventory_control}. It can be shown that $s^*, S^* \rightarrow \underline{a}$ as $K \downarrow 0$.}
\label{table_impulse}
\end{table}

\subsubsection{Brief remarks on the cases of Examples \ref{example_dividedn_fixed_costs} and \ref{example_dividedn_fixed_costs_dual}.
} \label{remark_other_impulse} As in the singular control case, verification is in general harder for the spectrally negative case than for the spectrally positive case.  

For Example \ref{example_dividedn_fixed_costs_dual}, the variational inequalities \eqref{QVI_system}  can be shown without much difficulty.  Similarly to Example \ref{example_inventory_control} above, the generator part of \eqref{QVI_system} holds trivially; this is due to the fact that in this case the controlling region is $(-\infty, s^*)$ and the waiting region is $(s^*, 0]$; the process does not jump from the former to the latter and hence the results similar to Lemma \ref{lemma_easy_results_impulse} hold.  The other parts of \eqref{QVI_system} can be shown using the log-concavity of the scale function as in Section \ref{section_log_concavity}, which essentially shows that $-v_{s^*, S^*}'(x) < -C_U$ if and only if $x \in (s^*, S^*)$; see Lemma 5.3 of \cite{Bayraktar_2013}.

On the other hand, the verification for Example \ref{example_dividedn_fixed_costs} can only be done for a subset of spectrally negative \lev processes.  This is again due to the fact, in this case, that the controlling region is $(s^*, \infty)$ and the waiting region is $[0,s^*)$; the process can jump from the former to the latter, where the form of $v_{s^*, S^*}$ changes.

\section{Zero-sum games between two-players} \label{section_game}

In this section, we consider optimal stopping games between two players: the \emph{inf player} and the \emph{sup player}, whose strategies are given by stopping times $\theta$ and $\tau$, respectively.
Here, a common expected payoff is minimized by the former and is maximized by the latter.  
The problem is terminated at the time either of the two players decides to stop or
at the first exit time from some closed interval $\mathcal{I}$:
\[T_{\mathcal{I}^c}:=\inf\{\,t >  0\,:\, X_t \notin \mathcal{I} \,\}.\]
Without loss of generality, these can be assumed to satisfy
\begin{align}
\theta, \tau \leq T_{\mathcal{I}^c}, \quad a.s. \label{tau_sigma_bounded}
\end{align}

Let $q > 0$ be the discount factor and  the terminal payoff be given by
\begin{enumerate}
\item $g_I$: when the inf player stops first,
\item $g_S$:  when the sup player stops first,
\item $g$: when both players stop simultaneously (including the case $\theta = \tau = T_{\mathcal{I}^c}$),
\end{enumerate}
such that $g(x) = 0$ for $x \notin \mathcal{I}$.
Then given any pair of strategies
$(\theta, \tau)$, the expected cost (resp.\ reward) for the inf (resp.\ sup) player is
\begin{align}
v(x; \theta, \tau) :=\E_x \Big[  1_{\{\theta < \tau \}} e^{- q \theta} g_I(X_\theta)  + 1_{\{\tau < \theta \}} e^{- q \tau} g_S(X_\tau)  + 1_{\{\tau = \theta < \infty  \}} e^{- q \tau} g(X_\tau)   
\Big].\label{def_V}
\end{align} 
The objective is to determine, if it exists, a pair of stopping times $(\theta^*, \tau^*)\subset \mathcal{S}$, called the \emph{saddle point}, that constitutes the \emph{Nash equilibrium}:
\begin{align}
v(x; \theta^*, \tau) \leq v(x; \theta^*, \tau^*)  \leq v(x; \theta, \tau^*), \quad \forall \, \theta, \tau \in \S, \label{saddle_pt}
\end{align}
where $\mathcal{S}$ is  the set of stopping times satisfying \eqref{tau_sigma_bounded}.

\begin{example} \label{example_CDS_game}Egami et al.\ \cite{Leung_Yamazaki_2011} considered several games in the setting of a credit default swap (CDS) contract as extensions to the optimal stopping problem considered in Leung and Yamazaki \cite{leung_yamazaki_2013}.

As in a usual perpetual CDS contract, the sup player (protection buyer) pays premium continuously and whenever the default event $\{X < 0 \}$ happens, the sup player receives from the inf player (seller) a fixed default payment $1$, and the contract is terminated. 

In their \emph{cancellation game}, they added a feature that the sup player and inf player both have an option to cancel the contract before default  for a fee, whoever cancels first.  
Specifically, 
\begin{enumerate}
\item the  sup player begins by paying premium at rate $p$ over time  for a notional amount $1$ to be paid at default;
\item prior to default, the sup player and the inf player can select a time to cancel the contract;
\item when the sup player cancels, he is incurred the fee  $\costb$ to be paid to the inf player; when the inf player cancels, he is incurred $\costs$ to be paid to the sup player;
\item if the
sup player and the inf player exercise simultaneously, then  both pay the fee upon exercise. 
\end{enumerate}
For the game to make sense, these parameters are assumed to satisfy 
\begin{align}
\acheck >  \costs\geq 0, \quad \pcheck >0, \quad \costb +\costs >0.
\end{align}

Namely, the inf player wants to minimize while the sup player wants to maximize the common expectation:
\begin{multline}
V(x; \theta, \tau) :=\E_x \left[ -\int_0^{\tau \wedge \theta} e^{-qt} p\,\diff t  \right. \\ \left. + 1_{\{\tau \wedge \theta < \infty\}}\bigg(  e^{-q T_{(-\infty,0)}} 1_{\{{\tau = \theta} =T_{(-\infty,0)} \}} + 1_{\{\tau \wedge \theta < T_{(-\infty,0)} \}} e^{-q (\tau \wedge \theta)}  \left(- \costb 1_{\{\tau \leq \theta \}} +  \costs 1_{\{\tau \geq \theta \}}  \right) \bigg) \right], \label{def_V_cds}
\end{multline} 
by choosing stopping times $\theta$ and $\tau$, respectively.

Let \begin{align}
 C(x;p) &:=\E_x \left[ -\int_0^{{T_{(-\infty,0)}} } e^{-qt} p\, \diff t +  e^{-q {T_{(-\infty,0)}}}     \right] =  \left(\frac{p}{q} +1\right) \lap(x) -\frac{p}{q}, \quad x > 0, \label{cds} 
\end{align}
where, by \eqref{first_passage_time},  
\begin{align*}
\lap(x) &:= \E_x \left[  e^{-q{T_{(-\infty,0)}}}\right] = Z^{(q)}(x) -  \frac q {\Phi(q)} W^{(q)}(x), \quad x \in  \R. 
\end{align*}
Then, by the strong Markov property, \eqref{def_V_cds} can be written
\begin{align*}V(x; \theta, \tau) &= C(x;p)+ v(x; \theta, \tau), \quad x > 0, \end{align*}
where \begin{align}
v(x; \theta, \tau) &:= \E_x \left[ e^{-q (\tau\wedge \theta)} \left( g_S(X_{\tau}) 1_{\{\tau < \theta \}} + g_I(X_{ \theta}) 1_{\{\tau > \theta \}} + g(X_{\tau }) 1_{\{\tau = \theta \}} \right) 1_{\{\tau \wedge \theta < \infty\}} \right], \label{definition_v}
\end{align}
 with, for $x \in \R$, 
\begin{align}\label{hx}g_S(x) &:=  1_{\{x > 0 \}}  \Big[ \Big(\frac{p}{q} - \costb \Big) - \Big(\frac{p}{q} +1 \Big) \lap(x)\Big], \\
\label{gx}g_I(x) &:= 1_{\{x > 0 \}}  \Big[ \Big(\frac{p}{q} +\costs \Big) - \Big(\frac{p}{q} +1\Big) \lap(x)\Big], \\
\label{fx}g(x) &:= 1_{\{x > 0 \}}  \Big[ \Big(\frac{p}{q} - \costb +\costs \Big) - \Big(\frac{p}{q} +1\Big) \lap(x)\Big].
\end{align}
In other words, the problem is to identify the pair of strategies $(\theta^*, \tau^*)$ such that \eqref{saddle_pt} holds.

%
%

\end{example}

\subsection{Threshold strategies}  
If the (common) payoff functions have some monotonicity with respect to the position of $X$ as in the examples given in Section \ref{section_one_parameter}, it is expected that both implement threshold strategies where one of them stops when $X$ is sufficiently high while the other stops when it is sufficiently low.  Hence, it is a reasonable conjecture that the equilibrium is characterized by two boundaries: $\alpha < \beta$ or $\beta < \alpha$.

We shall now consider a pair of strategies $(\theta_\alpha, \tau_\beta)$ such that

\begin{enumerate}
\item if $\alpha < \beta$, then $\theta_\alpha := \inf \{ t > 0: X_t < \alpha \}$ and $\tau_\beta := \inf \{ t > 0: X_t > \beta \}$,
\item if $\beta < \alpha$, then $\theta_\alpha := \inf \{ t > 0: X_t > \alpha \}$ and $\tau_\beta := \inf \{ t > 0: X_t < \beta \}$.
\end{enumerate}
In order to satisfy the condition \eqref{tau_sigma_bounded}, we must have $\underline{\mathcal{I}} \leq  \alpha < \beta \leq \overline{\mathcal{I}}$ and $\underline{\mathcal{I}} \leq \beta < \alpha \leq \overline{\mathcal{I}}$ for (1) and (2), respectively.

In this case, the players' expected NPVs of reward/cost \eqref{def_V} becomes
\begin{align*}
v_{\alpha, \beta}(x) :=\E_x \Big[  1_{\{\theta_\alpha < \tau_\beta \}} e^{- q \theta_\alpha} g_I(X_{\theta_\alpha})  + 1_{\{\tau_\beta < \theta_\alpha \}} e^{- q \tau_\beta} g_S(X_{\tau_\beta})   \Big].
\end{align*} 
By the reviewed results in Section \ref{fluctuations_underlying}, this can be computed by the scale function and the \lev measure.  

Focusing on strategy pairs given by $(\theta_\alpha, \tau_\beta)$, the first step again is to choose a candidate barrier pair $(\alpha^*, \beta^*)$ using two equations.
The expected degree of smoothness is the same as the impulse control case (see Section \ref{smoothness_impulse})  and is one less than the singular control case (see Section \ref{smoothness_double_barrier}).   More precisely, we have the following for the case $\alpha^* < \beta^*$ (the case $\beta^* < \alpha^*$ holds in the same way by swapping the roles of $\alpha^*$ and $\beta^*$):
\begin{enumerate}
\item 
Regarding the smoothness of the value function at the lower barrier $\alpha^*$,
\begin{enumerate}
\item if $\alpha^*$ is regular for $(-\infty, \alpha^*)$ (or equivalently $X$ is of unbounded variation), then the continuous differentiability at $\alpha^*$ is expected;
\item if $\alpha^*$ is irregular for $(-\infty, \alpha^*)$ (or equivalently $X$ is of bounded variation), then the continuity at $\alpha^*$ is expected.
\end{enumerate}
\item Regarding the smoothness at the upper barrier $\beta^*$, because it is always regular for $(\beta^*, \infty)$, continuous differentiability is expected at $\beta^*$ regardless of the path variation.
\end{enumerate}

\subsubsection{The case of Example \ref{example_CDS_game} } \label{game_existence_example}

In the cancellation game, the sup player has an incentive to cancel the contract when default  is less likely, or equivalently when $X$ is sufficiently high. On the other hand, the inf player tends to  cancel it when default is likely to occur, or equivalently when $X$ is sufficiently small. Because $\mathcal{I} = [0, \infty)$, we can conjecture that the sup player and the inf player choose the strategies $\tau_{\beta^*}$ and $\theta_{\alpha^*}$ for some values $0 \leq \alpha^* < \beta^* \leq \infty$.  Regarding the cases $\alpha^* = 0$ and $\beta^* = \infty$, see the interpretations given in Remark \ref{remark_alpha_zero}. 

For $0<\alpha<x < \beta <\infty$, it is straightforward to write
\begin{align} \label{delta_by_upsilon}
\begin{split}
v_{\alpha,\beta}(x) - g_S(x)  &=\Upsilon(x;\alpha,\beta) - \frac p q  + \costb, \\
v_{\alpha,\beta}(x) - g_I(x)  &=\Upsilon(x;\alpha,\beta) - \frac p q -  \costs,
\end{split}
\end{align}
where
\begin{align}
\begin{split}
\Upsilon(x; \alpha, \beta) &:=  - \costb \E_x\left[  e^{-q (\theta_\alpha \wedge \tau_\beta)} 1_{\{\tau_\beta < \theta_\alpha \}} \right]  +\costs \E_x\left[  e^{-q (\theta_\alpha \wedge \tau_\beta)} 1_{\{\tau_\beta > \theta_\alpha \, \textrm{or} \;  \theta_\alpha = \tau_\beta = T_{(-\infty, 0)} \}} \right] \\ &\; -\costs \E_x\left[  e^{-q (\theta_\alpha \wedge \tau_\beta)}1_{\{\theta_\alpha = \tau_\beta = {T_{(-\infty, 0)}} \}} \right]. \end{split} \label{definition_upsilon}
\end{align}
By the results in Section \ref{fluctuations_underlying} together with the compensation formula (see Theorem 4.4 of  \cite{Kyprianou_2006}), we can write
\begin{align}
\begin{split}
\Upsilon(x;\alpha,\beta) &=  W^{(q)}(x-\alpha) \frac {\Lambda(\alpha,\beta)} {W^{(q)}(\beta-\alpha)} - \Lambda(\alpha,x) + \frac p q - \gamma_S 
, \quad \beta > x > \alpha > 0,
\end{split}  \label{upsilon_in_terms_of_scale_function}
\end{align}
 where, for $0 < \alpha < \beta < \infty$,
\begin{align}
\Lambda(\alpha,\beta) &:= \frac p q - \costb  - \Big(\frac p q + \costs \Big)  Z^{(q)} (\beta-\alpha) +  \frac {1-\costs} q \int_{(-\infty, -\alpha)}  \left( Z^{(q)}(\beta-\alpha) - Z^{(q)}(\beta+u) \right) \nu(\diff u) . \label{def_Psi} 
\end{align}
We also define the derivative of \eqref{def_Psi} as, for $0 < \alpha < \beta < \infty$,
\begin{align*}
\lambda(\alpha, \beta) &:= \frac \partial {\partial \beta}\Lambda(\alpha,\beta) \\ &=   - \left( \pcheck +\costs q \right) W^{(q)} (\beta-\alpha)   + \left( \acheck - \costs \right) \int_{(-\infty, -\alpha)}  \left( W^{(q)}(\beta-\alpha) - W^{(q)}(\beta+u) \right) \nu(\diff u) .
\end{align*}

We begin with establishing the continuous fit condition. First, by taking limits in \eqref{delta_by_upsilon}, we have, for $0 < \alpha < \beta < \infty$
\begin{align}
v_{\alpha,\beta}(\beta-)-g_S(\beta) &= \Upsilon(\beta-; \alpha, \beta) + \costb  =0, \label{eq_continuous_fit_B} \\
v_{\alpha,\beta}(\alpha+)-g_I(\alpha) &=W^{(q)}(0) \frac {\Lambda (\alpha,\beta)} {W^{(q)}(\beta-\alpha)}. \label{eq_continuous_fit_A}
\end{align}
This means that  continuous fit holds automatically at $\beta$. On the other hand, at $\alpha$, while continuous fit holds automatically for the case of unbounded variation, it holds if and only if 
\begin{align}
\mathfrak{C}_\alpha: \frac {\Lambda (\alpha,\beta)} {W^{(q)}(\beta-\alpha)} = 0 \label{cond_at_alpha}
\end{align} 
for the bounded variation case.

Now, by taking the derivative of \eqref{upsilon_in_terms_of_scale_function}, we obtain, for $\alpha < x < \beta$,
\begin{align*}
v_{\alpha,\beta}'(x+)-g_S'(x) = v_{\alpha,\beta}'(x+)-g_I'(x) = \Upsilon'(x+;\alpha,\beta) = W^{(q)\prime}((x-\alpha)+) \frac {\Lambda (\alpha,\beta)} {W^{(q)}(\beta-\alpha)}  - \lambda(\alpha,x). 
\end{align*}
Hence, the smooth fit at $\beta$ holds if and only if 
\begin{align*}
\mathfrak{C}_\beta:  W^{(q)\prime}((\beta-\alpha)-) \frac {\Lambda (\alpha,\beta)} {W^{(q)}(\beta-\alpha)}  - \lambda(\alpha, \beta)  = 0.
\end{align*}
Assuming that it has paths of unbounded variation ($W^{(q)}(0)=0$), then we obtain
\begin{align*}
v_{\alpha,\beta}'(\alpha+)-g'(\alpha)  &= W^{(q)\prime}(0+)  \frac {\Lambda (\alpha,\beta)} {W^{(q)}(\beta-\alpha)}, \quad 0 < \alpha < \beta.
\end{align*}
Therefore, $\mathfrak{C}_\alpha$ is also a sufficient condition for smooth fit at $\alpha$ for the unbounded variation case.  In addition, if $\mathfrak{C}_\alpha$ holds, then $\mathfrak{C}_\beta$ simplifies to
\begin{align*}
\mathfrak{C}_\beta':   \lambda(\alpha, \beta)  = 0.
\end{align*}

We conclude that
\begin{enumerate}
\item if $(\alpha^*, \beta^*)$ satisfy $\mathfrak{C}_\alpha$, then continuous fit at $\alpha^*$ holds for the bounded variation case and both continuous and smooth fit  at $\alpha^*$ holds for the unbounded variation case;
\item if $(\alpha^*, \beta^*)$ satisfy $\mathfrak{C}_\beta$, then both continuous and smooth fit conditions at $\beta^*$  hold for all cases.
\end{enumerate}
\begin{remark}
Note that the conditions $\mathfrak{C}_\alpha$ and $\mathfrak{C}_\beta$ (or $\mathfrak{C}'_\beta$) are the same as $\mathfrak{C}_a$ and $\mathfrak{C}_b$ (or $\mathfrak{C}_b'$) as in  \eqref{smoothness_condition1} and \eqref{smoothness_condition2} (or \eqref{smoothness_condition2_prime}) in the two-sided singular control case and are similar to $\mathfrak{C}_s$ and $\mathfrak{C}_S$ (or $\mathfrak{C}_S'$) as in \eqref{G_zero} and \eqref{H_zero} (or \eqref{H_zero_simple}) in the impulse control case, except that the form of $\Lambda$ is different.
\end{remark}



In order to show the existence of a pair that satisfy $\mathfrak{C}_\alpha$ and $\mathfrak{C}_\beta$, consider the function, for $0 < \alpha < \beta$,
\begin{align*}
\widehat{\lambda}(\alpha,\beta)  := \frac {\lambda (\alpha,\beta)} {W^{(q)}(\beta-\alpha)} = -\left( \pcheck + q \costs  \right) +  \big( \acheck - \costs \big)\int_{(-\infty, -\alpha)}  \left( 1 - \frac {W^{(q)}(\beta+u)} {W^{(q)}(\beta-\alpha)} \right) \nu (\diff u).
\end{align*}
By using the log-concavity of the scale function as in Section \ref{section_log_concavity}, the following can be easily derived.

\begin{lemma} \label{remark_hat_psi_monotonicity}
\begin{enumerate}
\item  For fixed $0 < \beta < \infty$, $\alpha \mapsto \widehat{\lambda}(\alpha,\beta)$ is decreasing  on $(0,\beta)$.
\item For fixed $\alpha > 0$, $\beta \mapsto \widehat{\lambda}(\alpha,\beta)$ is decreasing on $(\alpha,\infty)$.
\end{enumerate}
\end{lemma}

Using Lemma \ref{remark_hat_psi_monotonicity}(2) and \eqref{scale_function_asymptotics}, for $\alpha > 0$, we can extend $\widehat{\lambda}(\alpha, \beta)$ to the cases $\beta = \alpha$ and $\beta = \infty$ with
\begin{align*}
\widehat{\lambda}(\alpha) &\equiv \widehat{\lambda}(\alpha,\alpha+) := \lim_{\beta \downarrow \alpha}\widehat{\lambda}(\alpha,\beta)  = -\left( \pcheck + q \costs \right)  + (\acheck - \costs) \bar{\nu}(\alpha),  \\
\widehat{\lambda}(\alpha,\infty) &:=  \lim_{\beta \rightarrow \infty}\widehat{\lambda}(\alpha,\beta) = -\left( \pcheck + q \costs \right)  +  (\acheck - \costs)  \Phi(q) \Psi (\alpha; \bar{\nu}) = \Phi(q) \Psi (\alpha; \widehat{\lambda}),
\end{align*}
where 
\begin{align*}
\bar{\nu}(x) := \nu(-\infty, - x), \quad x > 0.
\end{align*}
We shall see that the function $\widehat{\lambda} (\cdot)$ plays the same role as $\tilde{f}' (\cdot)$ in Examples \ref{example_two_sided_control} and \ref{example_inventory_control}. Because $\widehat{\lambda}(\cdot)$ and $\Psi(\cdot, \widehat{\lambda})$ are monotonically decreasing, we can define $\overline{\alpha} := \overline{a}(-\widehat{\lambda})$ and $\underline{\alpha} := \underline{a}(-\widehat{\lambda})$ as in Definitions \ref{def_a_bar} and \ref{def_a_bar_under}, respectively.  These will serve as bounds on $\alpha^*$ and we will have $\underline{\alpha} \leq \alpha^* < \overline{\alpha}$.

Egami et al.\ \cite{Leung_Yamazaki_2011} show that there always exists a pair $(\alpha^*,\beta^*)$ belonging to one of the following four cases:
\begin{description}
\item[case 1] $0 < \alpha^*  < \beta^* < \infty$; 
\item[case 2] $0 < \alpha^* < \beta^* = \infty$; 
\item[case 3] $0 = \alpha^*  < \beta^* < \infty$; 
\item[case 4] $0 = \alpha^* < \beta^* = \infty$;
\end{description}
which satisfy $\mathfrak{C}_\alpha$ when $\alpha^* > 0$ and  $\mathfrak{C}_\beta$ when $\beta^* < \infty$.


Here, we only give a brief sketch of the proof that if
\begin{align}
\underline{\alpha} > 0 \quad \textrm{and} \quad \sup_{\beta > \underline{\alpha}} \Lambda(\underline{\alpha}, \beta) > 0, \label{assump_case1}
\end{align}
 then \textbf{case 1}  holds.  (If these are violated, $\alpha^* = 0$ and/or $\beta^* = \infty$; see Remark \ref{remark_alpha_zero} below.)  
To this end, observe that
\begin{align}
\begin{split}
\frac \partial {\partial \alpha}\Lambda(\alpha,\beta)  
&= - W^{(q)} (\beta-\alpha) \widehat{\lambda}(\alpha) 
\end{split}
\label{Psi_derivative_A}
\end{align}
 is negative for every $\alpha \in (0,\overline{\alpha})$ by how $\overline{\alpha}$ is chosen  as in Definition \ref{def_a_bar}. Hence, the function $\alpha \mapsto \sup_{\beta > \alpha} \Lambda(\alpha, \beta)$ is monotonically decreasing on $(0, \overline{\alpha})$.  Thanks to the continuity of $\Lambda(\alpha, \beta)$ and \eqref{assump_case1}, 
 if we can show that 
  $\sup_{\beta > \overline{\alpha}}\Lambda(\overline{\alpha}, \beta) < 0$, then there must exist $\alpha^* \in (\underline{\alpha}, \overline{\alpha})$ such that $\sup_{\beta > \alpha^*}\Lambda(\alpha^*, \beta) = 0$ with its local maximum attained at $\beta^*$.
%
%
 Indeed, by Lemma \ref{remark_hat_psi_monotonicity}(2) and how $\overline{\alpha}$ is chosen, $\widehat{\lambda}(\overline{\alpha}, \beta) \leq 0$ or equivalently $\lambda(\overline{\alpha}, \beta)  \leq 0$ for $\beta \in (\overline{\alpha}, \infty)$ and hence $\sup_{\beta > \overline{\alpha}}\Lambda(\overline{\alpha}, \beta) = \Lambda(\overline{\alpha}, \overline{\alpha}+) = - (\gamma_I + \gamma_S)< 0$.
 
These  properties of the shapes of $\lambda$ and $\Lambda$ can be confirmed by the numerical plots given in Figure \ref{figure_psi}.

\begin{remark} \label{remark_alpha_zero}While the details are omitted in this note,  when \eqref{assump_case1} does not hold, necessarily $\alpha^* = 0$ and/or $\beta^* = \infty$.  In the latter case, it can be shown that the sup player never stops in the equilibrium. 

In the case $\alpha^* = 0$, it may not yield the Nash equilibrium for the unbounded variation case.  To see this, we notice that a default happens as soon as $X$ goes below zero. Therefore, in the event that $X$ continuously passes (creeps) through  zero, the inf player would optimally seek to exercise  at a level as close to zero as possible. Nevertheless, this timing strategy is not admissible, though it can be approximated arbitrarily closely by admissible  stopping times. It can be shown that $\alpha^*=0$ is possible only if  the jump part $X^d$ of $X$ is of bounded variation.  


\end{remark}

\begin{figure}[htbp]
\begin{center}
\begin{minipage}{1.0\textwidth}
\centering
\begin{tabular}{cc}
\includegraphics[scale=0.4]{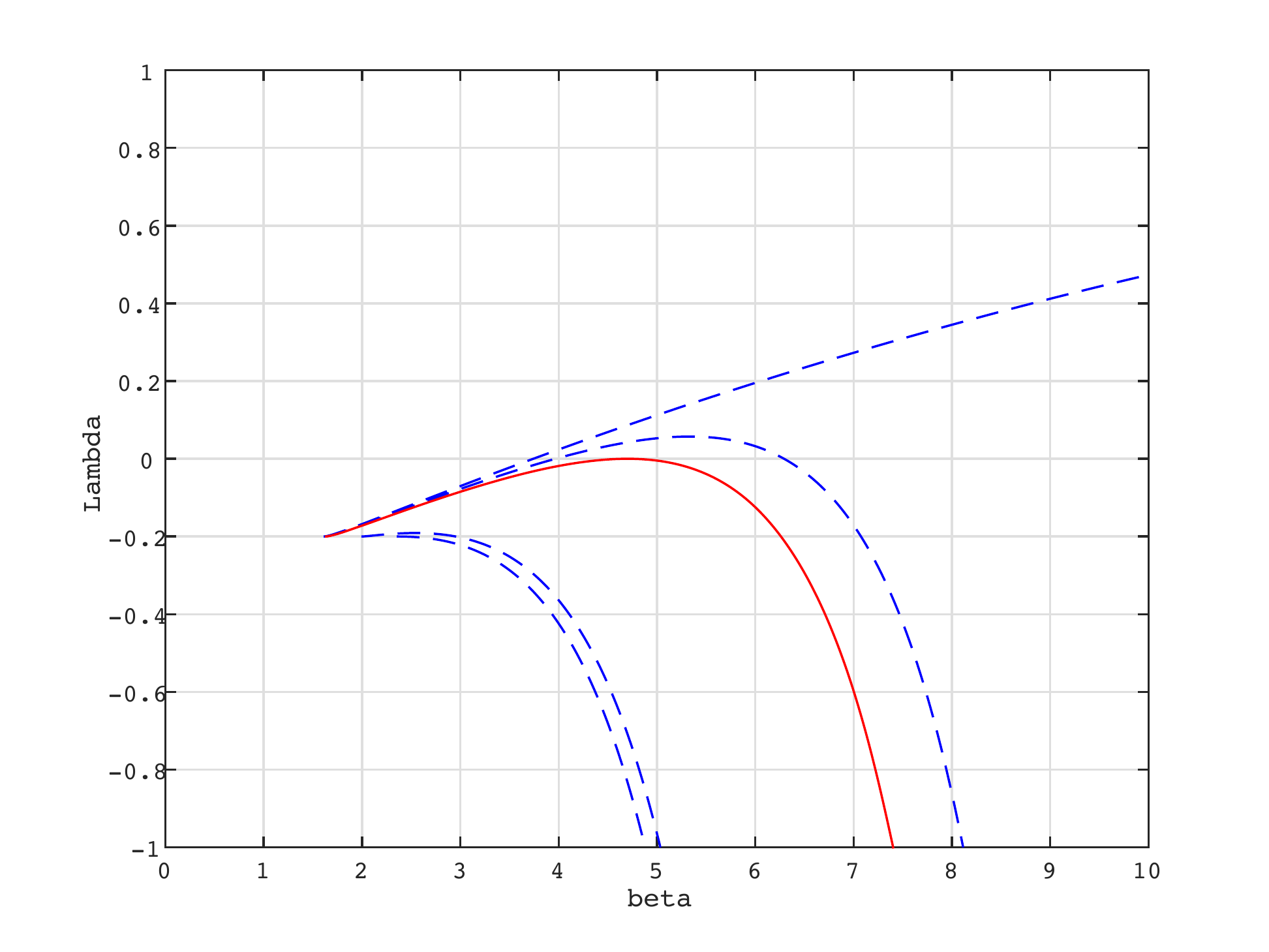}  & \includegraphics[scale=0.4]{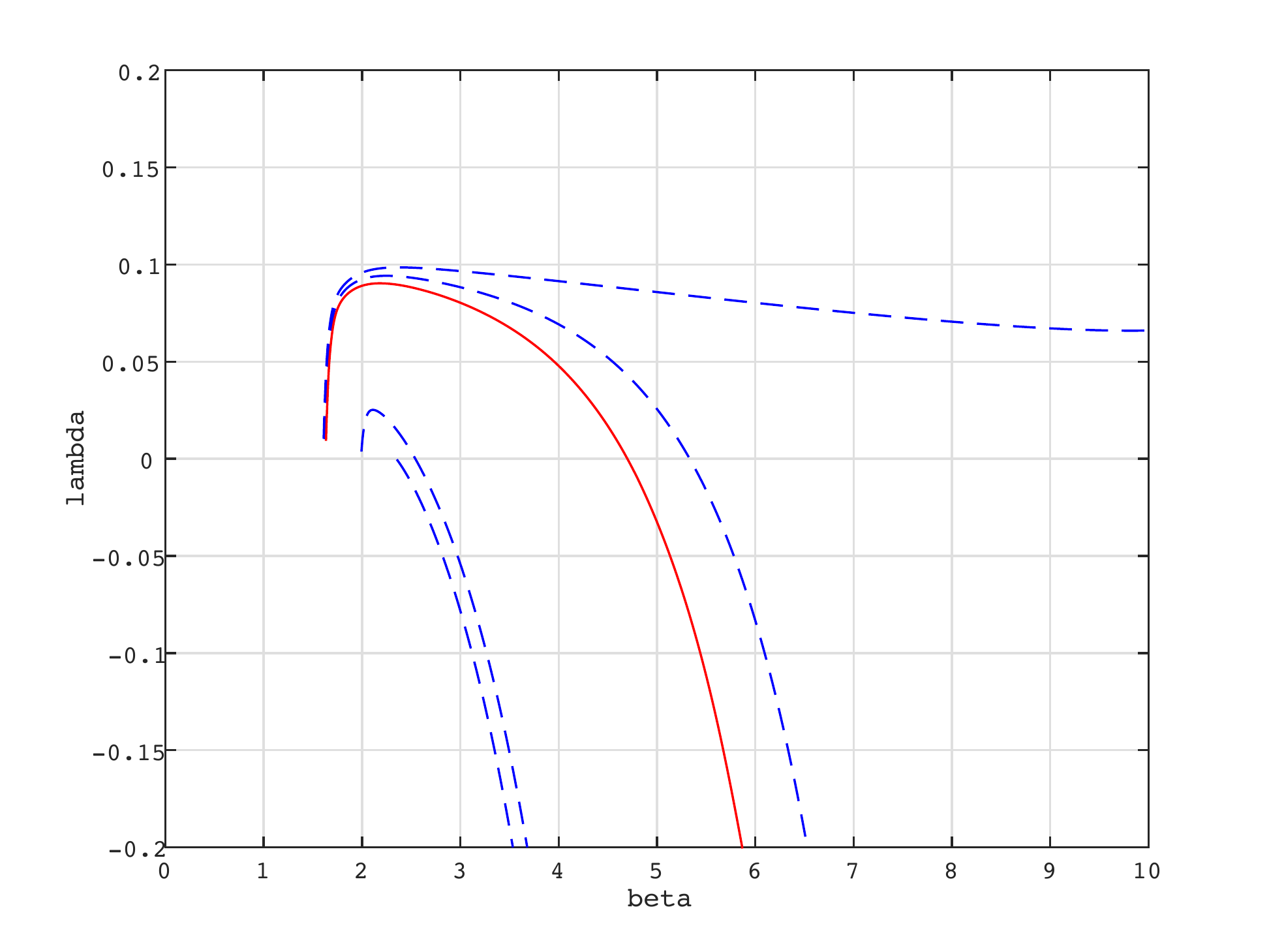} \\
$\beta \mapsto \Lambda(\alpha, \beta)$  & $\beta \mapsto \lambda(\alpha, \beta)$ 
\end{tabular}
\end{minipage}
\caption{Existence of $(\alpha^*, \beta^*)$ for Example \ref{example_CDS_game}. Plots of $\beta \mapsto \Lambda(\alpha, \beta)$ on $[\alpha, \infty)$ for the starting values $\alpha = \underline{\alpha}, (\underline{\alpha} + \alpha^*)/2, \alpha^*, (\alpha^* + \overline{\alpha})/2, \overline{\alpha}$.  The solid curve in red corresponds to the one for $\alpha = \alpha^*$; the point at which $\Lambda(\alpha^*, \cdot)$ is tangent to the x-axis (or $\lambda(\alpha^*, \cdot)$ vanishes) becomes $\beta^*$. The function $\Lambda(\underline{\alpha},\cdot)$ is monotonically increasing while $\Lambda(\overline{\alpha},\cdot)$ is monotonically decreasing. Equivalently, $\lambda(\underline{\alpha},\cdot)$ is uniformly positive while $\lambda(\overline{\alpha},\cdot)$ is uniformly negative. }  \label{figure_psi}
\end{center}
\end{figure}

\subsection{Variational inequalities and verification} \label{variational_inequality_games}
The verification of optimality (for both players) require that, when $\alpha^* < \beta^*$,
\begin{align} \label{variational_games}
\begin{split}
g_S(x) \leq v_{\alpha^*, \beta^*}(x)  &\leq g_I(x), \quad x \in \mathcal{I}, \\
(\mathcal{L}-q) v_{\alpha^*, \beta^*}(x)  &\geq 0, \quad  x  \in (-\infty, \alpha^*) \cap \mathcal{I}^o, 
\\
(\mathcal{L}-q) v_{\alpha^*, \beta^*}(x)   &= 0, \quad x \in (\alpha^*, \beta^*) \cap \mathcal{I}^o, 
\\
(\mathcal{L}-q) v_{\alpha^*, \beta^*}(x)  &\leq 0, \quad x \in (\beta^*, \infty) \cap \mathcal{I}^o. 
\end{split}
\end{align}
On the other hand, when $\alpha^* > \beta^*$, it requires that
\begin{align*}
g_S(x) \leq v_{\alpha^*, \beta^*}(x)  &\leq g_I(x), \quad x \in \mathcal{I}, \\
(\mathcal{L}-q) v_{\alpha^*, \beta^*}(x)  &\leq 0, \quad  x  \in (-\infty, \beta^*) \cap \mathcal{I}^o, \\
(\mathcal{L}-q) v_{\alpha^*, \beta^*}(x)   &= 0, \quad x \in (\beta^*, \alpha^*) \cap \mathcal{I}^o, \\
(\mathcal{L}-q) v_{\alpha^*, \beta^*}(x)  &\geq 0, \quad x \in (\alpha^*, \infty) \cap \mathcal{I}^o.
\end{align*}

Suppose $\alpha^* < \beta^*$.  From the inf player's perspective, assuming that the sup player's strategy is given by $\tau_{\beta^*}$ (so that the state space for the inf player is $\mathcal{I}_{\beta^*} := (-\infty, \beta^*) \cap \mathcal{I}$), the above variational inequalities satisfy those for the minimization problem for the inf player that
\begin{align*}
v_{\alpha^*, \beta^*}(x)  &\leq g_I(x), \quad x \in \mathcal{I}_{\beta^*}, \\
(\mathcal{L}-q) v_{\alpha^*, \beta^*}(x)  &\geq 0, \quad  x  \in (-\infty, \alpha^*) \cap \mathcal{I}_{\beta^*}^o, \\
(\mathcal{L}-q) v_{\alpha^*, \beta^*}(x)   &= 0, \quad x \in (\alpha^*, \beta^*). 
\end{align*}
Similarly, from the sup player's perspective, assuming that the inf player's strategy is given by $\theta_{\alpha^*}$ (so that the state space of the sup player is $\mathcal{I}_{\alpha^*} := (\alpha^*, \infty) \cap \mathcal{I}$), the above variational inequalities satisfy those for the maximization problem for the sup player that
\begin{align*}
v_{\alpha^*, \beta^*}(x)    &\geq g_S(x), \quad x \in \mathcal{I}_{\alpha^*},  \\
(\mathcal{L}-q) v_{\alpha^*, \beta^*}(x)  &\leq 0, \quad  x  \in (\beta^*, \infty) \cap \mathcal{I}_{\alpha^*}^o,\\
(\mathcal{L}-q) v_{\alpha^*, \beta^*}(x)   &= 0, \quad x \in (\alpha^*, \beta^*). 
\end{align*}
The case $\alpha^* > \beta^*$ is similar, and hence we omit the details.

This is a rough illustration on why these conditions are imposed for verification.  We refer the reader to \cite{Leung_Yamazaki_2011} and also \cite{ekstrom2008optimal, peskir2009optimal} for more rigorous arguments.    In general, if $v_{\alpha^*, \beta^*}$ is unbounded or $\mathcal{I}$ has a finite boundary at which $v_{\alpha^*, \beta^*}$ fails to be smooth/continuous, some localizing arguments are necessary.

%

\subsubsection{Verification for Example \ref{example_CDS_game}} \label{game_verification_example}

Here we shall illustrate a proof technique on how the candidate value function $v_{\alpha^*, \beta^*}$ solves the variational inequalities, focusing on  Example \ref{example_CDS_game} for the case $0 < \alpha^* < \beta^* < \infty$.

 \begin{figure}[htbp]
\begin{center}
\begin{minipage}{1.0\textwidth}
\centering
\begin{tabular}{c}
 \includegraphics[scale=0.4]{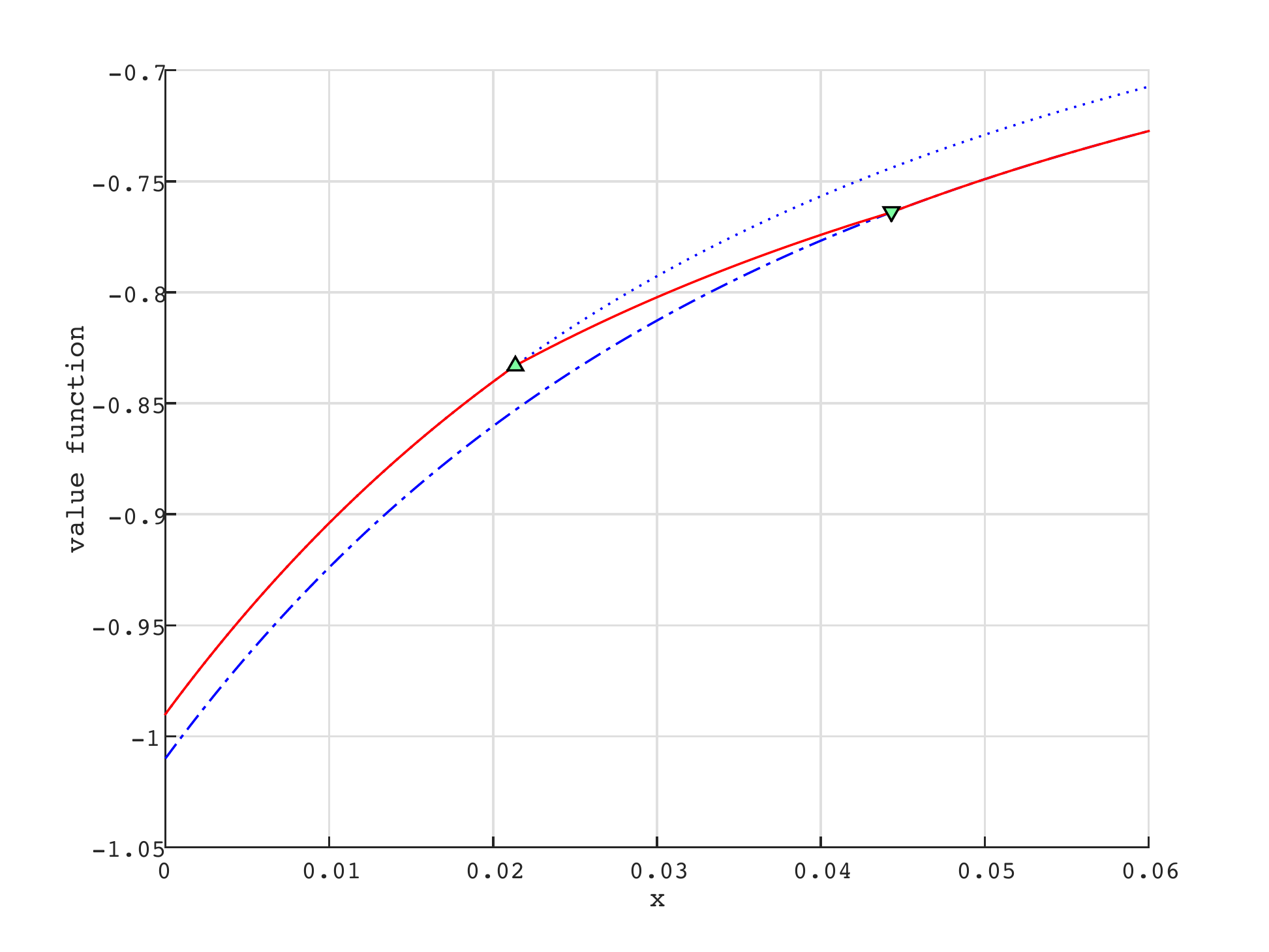}
 \end{tabular}
\end{minipage}
\caption{A sample plot of the value function $v_{\alpha^*, \beta^*}$ (solid red line) for Example \ref{example_CDS_game} when $X$ is of unbounded variation. The up-pointing and down-pointing triangles show the points at $\alpha^*$ and $\beta^*$, respectively. The two dotted lines show the stopping values $g_S$ and $g_I$.}   \label{figure_value_function_game}
\end{center}
\end{figure}

By \eqref{delta_by_upsilon}, we can write
\begin{align}
\begin{split}
v_{\alpha^*,\beta^*}(x) &= \left\{ \begin{array}{ll} g_S(x), & x \geq \beta^* \\
g_S(x) + (v_{\alpha^*,\beta^*}(x) - g_S(x)), & \alpha^* < x < \beta^* \\
g_I(x), & x \leq \alpha^*
\end{array}\right\} = - \Big(\frac{\pcheck}{q} +\acheck\Big) \zeta(x) + J(x)
\end{split} \label{value_function}
\end{align}
where
\begin{align} \label{def_J}
J(x) := \left\{ \begin{array}{ll} \frac{\pcheck}{q} - \costb, & x \geq \beta^*, \\
 \Upsilon(x; \alpha^*, \beta^*),  & \alpha^* < x < \beta^*, \\
\frac{\pcheck}{q} +\costs, & 0 \leq x < \alpha^*, \\ \frac{\pcheck}{q} +\acheck & x \leq 0.
\end{array}\right.
\end{align}
Here, by \eqref{cond_at_alpha},
\begin{align}
\begin{split}
\Upsilon(x;\alpha^*,\beta^*) &=  \Big(\frac p q + \costs \Big)  Z^{(q)} (x-\alpha^*) - \frac {1-\costs} q \int_{(-\infty, -\alpha^*)}  \left( Z^{(q)}(x-\alpha^*) - Z^{(q)}(x+u) \right) \nu(\diff u).
\end{split}  \label{upsilon_simplified}
\end{align}
See Figure \ref{figure_value_function_game} for a sample plot of the value function along with the stopping values.

Below, we show briefly that $v_{\alpha^*, \beta^*}$ solves \eqref{variational_games} when $0 < \alpha^* < \beta^* < \infty$.
\begin{lemma} \label{lemma_verification_games}
Suppose $W^{(q)}$ is sufficiently smooth on $(0, \infty)$ (i.e.\ $C^1$ when $X$ is of bounded variation and $C^2$ when it is of unbounded variation).  Then we have the following:
\begin{enumerate}
\item $g_S(x) \leq v_{\alpha^*, \beta^*}(x)  \leq g_I(x), \quad x \in [0, \infty)$, 
\item $(\mathcal{L}-q) v_{\alpha^*, \beta^*}(x)  \geq 0, \quad  x  \in (0, \alpha^*)$, 
\item $(\mathcal{L}-q) v_{\alpha^*, \beta^*}(x)   = 0, \quad x \in (\alpha^*, \beta^*)$, 
\item $(\mathcal{L}-q) v_{\alpha^*, \beta^*}(x)  \leq 0, \quad x \in (\beta^*, \infty)$. 
\end{enumerate}
\end{lemma}
\begin{proof}[Brief sketch of proof]
(1) We show for $x \in (\alpha^*, \beta^*)$; the other cases are immediate.  

The proof is relatively straightforward by  the log-concavity of the scale function as in Section \ref{section_log_concavity} and the shapes of $\Lambda$ and $\lambda$ given by 
\begin{align}
\Lambda(\alpha^*, \beta) \leq 0 \quad \textrm{and} \quad \lambda(\alpha^*, \beta) \geq 0, \quad \alpha^* <  \beta < \beta^*. \label{Psi_psi_negative}
\end{align}
Here \eqref{Psi_psi_negative} holds because, by Lemma \ref{remark_hat_psi_monotonicity},  $\beta \mapsto \Lambda(\alpha^*, \beta)$ increases on $(\alpha^*, \beta^*)$ and decreases on $(\beta^*, \infty)$ with its peak given at $\Lambda(\alpha^*, \beta^*) = 0$ (see Figure \ref{figure_psi}). 

Now, with the help of \eqref{Psi_derivative_A} and the log-concavity,
\begin{align*}
\frac {\partial_+} {\partial_+ \alpha} (v_{\alpha, \beta^*}(x) - g_I(x))   = \Big[ \frac {\partial_+} {\partial_+ \alpha} \frac {W^{(q)}(x-\alpha)} {W^{(q)}(\beta^*-\alpha)} \Big] \Lambda(\alpha,\beta^*) > 0, \quad \alpha^* < \alpha < x < \beta^*.
\end{align*}
Hence, by \eqref{eq_continuous_fit_A} and \eqref{Psi_psi_negative}, $0 \geq {W^{(q)}(0)} \Lambda(x,\beta^*) / W^{(q)} (x-\beta^*) = v_{x,\beta^*}(x+) - g_I(x) \geq  v_{\alpha^*,\beta^*}(x) - g_I(x)$ for $\alpha^* < x < \beta^*$.

On the other hand, by \eqref{Psi_psi_negative},
\begin{align*}
\frac {\partial_+} {\partial_+ \beta} (v_{\alpha^*,\beta}(x) - g_S(x)) 
 &= \frac {W^{(q)}(x-\alpha^*)} {(W^{(q)}(\beta-\alpha^*))^2} \big[ \lambda(\alpha^*,\beta)  W^{(q)}(\beta-\alpha^*) - \Lambda(\alpha^*,\beta) W^{(q) \prime}((\beta-\alpha^*)+)\big] \\
&> 0, \quad  \alpha^* < x < \beta < \beta^*.
\end{align*}
Therefore, by \eqref{eq_continuous_fit_B},
$0 = v_{\alpha^*,x}(x-) - g_S(x)  \leq  v_{\alpha^*,\beta^*}(x) - g_S(x)$ for $\alpha^* < x  < \beta^*$.


(2) By the assumption that $W^{(q)}$ is sufficiently smooth, the identity \eqref{martingale_W}  holds, and therefore
\begin{align}
(\mathcal{L}-q)\zeta(x) = 0, \quad x > 0. \label{zeta_harmonicity}
\end{align}
Hence,
\begin{align}
(\mathcal{L}-q) v_{\alpha^*,\beta^*}(x)  =  ( \acheck - \costs ) \bar{\nu}(x)  - (q \costs + \pcheck) =  \widehat{\lambda}(x). \label{generator_case_1}
\end{align}
Because $x < \alpha^* < \overline{\alpha}$, this must be positive by how $\overline{\alpha}$ is chosen.

(3) In view of \eqref{value_function}, \eqref{def_J}, and \eqref{upsilon_simplified}, it is immediate by  \eqref{martingale_Z_R} together with \eqref{zeta_harmonicity}. 

(4) This is as usual the hardest part because the process can jump from the stopping region of the sup player $(\beta^*, \infty)$ to the other two regions $(-\infty, \alpha^*)$ and $(\alpha^*, \beta^*)$, where the form of $v_{\alpha^*, \beta^*}$ changes. However, it is more straightforward than the two-sided singular control case that we studied in Section \ref{section_singular_control}.

In Egami et al.\ \cite{Leung_Yamazaki_2011}, they first show that $(\mathcal{L}-q)v_{\alpha^*,\beta^*}(\beta^*+) \leq (\mathcal{L}-q)v_{\alpha^*,\beta^*}(\beta^*-) = 0$ using how $\alpha^*$ and $\beta^*$ are chosen so that $v_{\alpha^*, \beta^*}$ gets smooth/continuous at $\beta^*$.  It then remains to show that $x \mapsto (\mathcal{L}-q)v_{\alpha^*,\beta^*}(x) $ is decreasing on $(\beta^*, \infty)$.  In view of the decomposition \eqref{value_function} and also \eqref{zeta_harmonicity}, it is equivalent to showing that $(\mathcal{L}-q)J(x) $  is decreasing on $(\beta^*, \infty)$.  Indeed, because $J'=J''=0$ on $x > \beta^*$,
\begin{align*}
(\mathcal{L}-q)J(x) = \int_{(-\infty, \beta^*-x)} \left[ J(x+u) - \Big(\frac p q - \costb \Big) \right] \nu(\diff u)  - (p - q \costb), \quad x > \beta^*,
\end{align*} 
where the integrand is nonnegative and monotonically decreasing in $x$ and the set $(-\infty, \beta^*-x)$ is decreasing in $x$ as well.
\end{proof}

In Table \ref{table_game}, we summarize the functions and parameters that played major roles in the above analysis for Examples  \ref{example_CDS_game}.
\begin{table}[h]
\centering
\begin{tabular}{l l}
\hline
$\Lambda(\alpha,\beta)$ &$:= \frac p q - \costb  - \big(\frac p q + \costs \big)  Z^{(q)} (\beta-\alpha) +  \frac {1-\costs} q \int_{(-\infty, -\alpha)}  \left[ Z^{(q)}(\beta-\alpha) - Z^{(q)}(\beta+u) \right] \nu(\diff u) $  \\
$\widehat{\lambda}(\alpha)$ &$:= -( \pcheck + q \costs )  + (\acheck - \costs) \bar{\nu}(\alpha)$
\\ [0.5ex] 
\hline
$\underline{\alpha}$& $:=\underline{a}(- \widehat{\lambda})$ \\
$\leq \alpha^*$& $:=\alpha^*$ of $(\alpha^*, \beta^*)$ such that  $\mathfrak{C}_\alpha$ and $\mathfrak{C}_\beta$ hold simultaneously\\
$< \overline{\alpha}$& $:=\bar{a}(-\widehat{\lambda})$  \\
$< \beta^*$ & $:=\beta^*$ of $(\alpha^*, \beta^*)$ such that  $\mathfrak{C}_\alpha$ and $\mathfrak{C}_\beta$ hold simultaneously  \\
\hline
\end{tabular} \vspace{0.3cm}
\caption{Summary of the key functions and parameters in Example \ref{example_CDS_game}. It can be shown that $\alpha^* = \underline{\alpha}$ when $\beta^* = \infty$.}
\label{table_game}
\end{table}

\subsection{Other optimal stopping games}

There are many other existing games studied for a spectrally one-sided \lev process.  The following problems can be formulated as \eqref{saddle_pt}.  However, there are clear differences with the problem considered above.

\begin{example} \label{example_mckean}
The McKean optimal stopping game corresponds to the case $\mathcal{I} = \R$ with $g_S (x) = g(x) =  (K - e^x) \vee 0$ and $g_I = (K - e^x) \vee 0 + \delta$ for some $K, \delta > 0$. In other words, this is an extension of the American put option where the seller (inf player) can also exercise with an additional fee $\delta$.  This problem was solved by Baurdoux and Kyprianou \cite{Baurdoux2008} for a spectrally negative \lev process. It is required that $0 \leq \psi(1) \leq q$ for the solution to be nontrivial.
\end{example}

\begin{example}  \label{example_convertible}As a way to model a version of the convertible bond, Gapeev and K{\"u}hn \cite{gapeev2005perpetual} and Baurdoux et al.\  \cite{BauKyrianouPardo2011} considered the problem where the cost (resp.\ reward) for the inf (resp.\ sup) player is given by 
\begin{align*}
V(x; \theta, \tau) :=\E_x \Big[ \int_0^{\tau \wedge \theta} e^{-qt} \big( C_1 + C_2 e^{X_t} \big)\diff t  + 1_{\{\theta \leq \tau \}} e^{- q \theta} (e^{X_\theta} \vee K) + 1_{\{\tau < \theta \}} e^{- q \tau + X_\tau} \Big],
\end{align*} 
for $C_1 \geq 0$ and $C_2, K > 0$.
This can be easily transformed to the formulation given in the beginning of this section.  Indeed, by the strong Markov property, we can write $V(x; \theta, \tau)  = v(x; \theta, \tau) + F(x)$
where
\begin{align*}
F(x) &:= \E_x \Big[ \int_0^{\infty} e^{-rt} (C_1 + C_2 e^{X_t} ) \diff t \Big], \\
v(x; \theta, \tau) &:= \E_x \Big[  1_{\{\theta \leq \tau \}} e^{- q \theta} \big( e^{X_\theta} \vee K - F(X_\theta) \big) + 1_{\{\tau < \theta \}} e^{- q \tau} (e^{X_\tau} - F(X_\tau))\Big].
\end{align*}
Hence, solving this is equivalent to solving \eqref{def_V} with  $g_I(x) = g(x) = e^x \vee K - F(x)$, $g_S(x) = e^x - F(x)$, and $\mathcal{I} = \R$.

Gapeev and K{\"u}hn \cite{gapeev2005perpetual}  considered the case of a Brownian motion plus i.i.d.\  exponential jumps.
Baurdoux et al.\  \cite{BauKyrianouPardo2011} studied for a spectrally positive \lev process.
\end{example}

In these examples, while the fluctuation theory and scale function can be used as main tools, the above techniques described in this section may not be directly used.

In Example \ref{example_mckean}, Baurdoux and Kyprianou \cite{Baurdoux2008} showed that the equilibrium is given by either $\tau^* := \inf \{ t > 0: X_t < k^*\}$ and $\sigma^* = \infty$, or $\tau^* := \inf \{ t > 0: X_t < x^*\}$ and $\sigma^* := \inf \{ t > 0: X_t \in [\log K, y^*]\}$ for some thresholds $k^*, x^*$ and $y^*$.  While continuous/smooth fit can be used to identify these values, due to the critical barrier $\log K$, one does not observe the dependency between the two parameters that we have seen in this section.

In Example \ref{example_convertible}, as shown in \cite{gapeev2005perpetual} and \cite{BauKyrianouPardo2011}, the equilibrium is given by two up-crossing times where at least one of them is the first time $X$ goes above the critical barrier $\log K$.  Therefore, again one does not observe the dependency between the two parameters.

\subsection{When a stopper is replaced with a controller}

One can naturally consider the case where the stopper(s) are replaced with singular controller(s).

The game between a controller and a stopper has been studied by Hern\'andez-Hern\'andez et al.\ \cite{hernandez2015zero} for the case driven by a diffusion process, where they obtained general results on the verification lemma and gave some explicitly solvable examples.

The case driven by a spectrally one-sided \lev process is studied by Hern\'andez-Hern\'andez and Yamazaki \cite{hernandez2015games}, where they considered  the problem where a stopper maximizes and a controller minimizes the expected value of some monotone payoff.  They considered both the spectrally negative and positive cases. Not surprisingly, the solution procedures are similar to the ones illustrated in this note: the candidate barriers $(a^*,b^*)$, which separate the state space into the stopping, waiting, and controlling regions,  are chosen by continuous/smooth fit so that
\begin{enumerate}
\item the value function at the boundary for the controller is continuously differentiable (resp.\ twice continuously differentiable) if it is irregular (resp.\ regular) for the controlling region;
\item the value function at the boundary for the stopper is continuous (resp.\ continuously differentiable) if it is irregular (resp.\ regular) for the stopping region.
 \end{enumerate}
 The verification of optimality can be carried out by showing the verification lemma as in the one given in Section \ref{variational_inequality_games}.  As we have seen, many parts of the verification can be carried out without much effort.  However, the difficulty is again to show the sub/super harmonicity at the region where the process can jump instantaneously to  the other regions. To deal with this, Hern\'andez-Hern\'andez and Yamazaki \cite{hernandez2015games} applied similar techniques as the ones discussed in Sections \ref{verification_singular} and \ref{game_verification_example}.
 
 The game between two singular controllers is also of great interest.  Under a certain monotonicity assumption on the payoff function, it is expected that the optimally controlled process becomes the doubly reflected \lev process similarly to the two-sided singular control case we studied in Section \ref{section_singular_control}.  Hence, the candidate value function can be computed again using the scale function and is expected to preserve the same smoothness as those observed in Section \ref{section_singular_control}.  Consequently, the two boundaries can be chosen in essentially the same way.  The verification lemma can be easily obtained by modifying \eqref{VI_system_double_reflected}.  It is expected that many of the techniques used in Section \ref{section_singular_control} can be recycled.

\section*{Acknowledgements}

K. Yamazaki is in part supported by MEXT KAKENHI Grant Number  26800092.

\bibliographystyle{abbrv}
\bibliography{dual_model_bib}

\end{document}